\documentclass[reqno]{amsart}
\usepackage{amsmath}
\usepackage{amssymb}
\usepackage{amsfonts}

\newtheorem{theorem}{Theorem}
\theoremstyle{plain}
\newtheorem{acknowledgement}{Acknowledgement}

\newtheorem{definition}{Definition}

\newtheorem{lemma}{Lemma}

\newtheorem{proposition}{Proposition}
\newtheorem{remark}{Remark}

\DeclareMathOperator{\Div}{div}
\numberwithin{equation}{section}
 \numberwithin{theorem}{section}
 \numberwithin{proposition}{section}
 \numberwithin{remark}{section}
 \numberwithin{definition}{section}
 \numberwithin{lemma}{section}
 \numberwithin{corollary}{section}
 \numberwithin{example}{section}
 \numberwithin{claim}{section}
\input{tcilatex}

\begin{document}
\title[Homogenization of stochastic wave equations]{Sigma-convergence of
semilinear stochastic wave equations}
\author{Gabriel Deugoue}
\address{G. Deugoue, Department of Mathematics and Computer Science,
University of Dschang, P.O. Box 67, Dschang, Cameroon}
\email{agdeugoue@yahoo.fr}
\author{Jean Louis Woukeng}
\address{J. L. Woukeng, Department of Mathematics and Computer Science,
University of Dschang, P.O. Box 67, Dschang, Cameroon}
\curraddr{J. L. Woukeng, Interdisciplinary Center for Scientific Computing
(IWR), University of Heidelberg, Im Neuenheimer Feld 205, 69120 Heidelberg,
Germany}
\email{jwoukeng@yahoo.fr}
\date{April, 2017}

\begin{abstract}
We address the homogenization of a semilinear hyperbolic stochastic partial
differential equation with highly oscillating coefficients, in the context
of ergodic algebras with mean value. To achieve our goal, we use a suitable
variant of the sigma-convergence concept that takes into account both the
random and deterministic behaviours of the phenomenon modelled by the
underlying problem. We also provide an appropriate scheme for the
approximation of the effective coefficients. To illustrate our approach, we
work out some concrete problems such as the periodic homogenization problem,
the almost periodic and the asymptotically almost periodic ones.
\end{abstract}

\maketitle

\section{Introduction and the main results\label{sec1}}

The need for taking random fluctuations into account in the study of complex
systems and physical phenomena resulting from the modeling to predictions is
now widely recognized by scientific community. Wave propagation described by
hyperbolic partial differential equations is one of the typical physical
phenomena widely observed in the nature, and has been studied over the years
and continue to attract the attention of scientists aiming at understanding
some physical phenomena such as sonic booms and bottleneck in traffic flows.
However due to the presence of turbulence, the more realistic way to model
and capture physical features of natural phenomena at large scale is to
introduce stochastic models. Stochastic partial differential equations
(SPDEs) are the most convenient mathematical models arising from modeling of
complex systems undergoing random influences.

Our aim in the current work is to analyze such a model represented by a
semilinear stochastic wave equation that can be used to study some problems
in nonlinear optics or the ones related to wave motion through the ocean or
the atmosphere. To this end, the problem we address is stated as follows.

Let $Q$ be a Lipschitz domain of $\mathbb{R}^{N}$ and $T$ a positive real
number. By $Q_{T}$ we denote the cylinder $Q\times(0,T)$. Let $\left(
\Omega, \mathcal{F},\{\mathcal{F}_{t}\}_{t\ge0},\mathbb{P},
\{W^{k}\}_{k\ge1}\right) $ a stochastic basis, that is a filtered
probability space with $\{W^{k}\}_{k\ge1}\}$ a sequence of independent
standard one dimensional Brownian motions relative to $\mathcal{F}_{t}$. Fix
a separable Hilbert space $\mathcal{U}$ with an associated orthonormal basis 
$(e_{k})_{k\ge1}$. We may define a cylindrical Wiener process by setting $%
W=\sum_{k=1}^{\infty}W^{k}e_{k}$ (see \cite{DaPrato}). By $L_{2}(\mathcal{U}%
, X)$ we denote the space of Hilbert-Schmidt operators from $\mathcal{U}$ to
a Hilbert space $X$. We also define the auxiliary space $\mathcal{U}%
_{0}\supset\mathcal{U}$ via $\mathcal{U}_{0}=\{v=\sum_{k\ge1}%
\alpha_{k}e_{k}:\sum_{k\ge1}\alpha_{k}^{2}k^{-2}<\infty\}$, endowed with the
norm $|v|^{2}_{\mathcal{U}_{0}}=\sum_{k=1}^{\infty}\alpha_{k}^{2}k^{-2}$,
for $v=\sum_{k\ge1}\alpha_{k}e_{k}$. It is a well known fact that there
exists $\Omega^{\prime}\in\mathcal{F}$ with $\mathbb{P}(\Omega^{\prime})=1$
such that $W(\omega)\in C(0,T;\mathcal{U}_{0})$ for any $\omega\in\Omega^{%
\prime}$ (see \cite{DaPrato}).

We consider the following semilinear stochastic hyperbolic initial value
problem 
\begin{equation}
\begin{cases}
du_{\varepsilon}^{\prime}-\Div\left( A_{0}^{\varepsilon}(x)\nabla
u_{\varepsilon}\right) dt=f(\frac{x}{\varepsilon},\frac{t}{\varepsilon }%
,u_{\varepsilon})dt+g(\frac{x}{\varepsilon},\frac{t}{\varepsilon }%
,u_{\varepsilon})dW\text{ in }Q_{T}, \\ 
u_{\varepsilon}=0\text{ on }\partial Q\times(0,T), \\ 
u_{\varepsilon}(x,0)=u^{0}(x)\text{ and }u_{\varepsilon}^{%
\prime}(x,0)=u^{1}(x)\text{ in }Q,%
\end{cases}
\label{1}
\end{equation}
where $\varepsilon>0$ is sufficiently small and $u_{\varepsilon}^{\prime}$
is the time derivative of $u_{\varepsilon}$ $(u_{\varepsilon}^{\prime}=\frac{%
\partial u_{\varepsilon}}{\partial t})$. We assume that the coefficients of (%
\ref{1}) are constrained as follows:

\textbf{(A1)~Uniform ellipticity}. \text{The function} $A_{0}^{\varepsilon}$ 
\text{defined by} $A_{0}^{\varepsilon}(x)=A_{0}(x,x/\varepsilon)$ \text{%
satisfies} $A_{0}\in\mathcal{C(}\bar{Q};L^{\infty}\left( \mathbb{R}%
_{y}^{N}\right) )^{N\times N}$ and is a $N\times N$ symmetric matrix
satisfying the following assumptions 
\begin{equation}
A_{0}\eta\cdot\eta\geq\alpha|\eta|^{2}
\end{equation}
for all $\eta\in\mathbb{R}^{N}$ and a.e. in $\mathbb{R}^{N}$, where $\alpha>0
$ is a given constant not depending on $x,t,y$ and $\eta$.\newline

\textbf{(A2)~Lipschitz continuity of }$f$. The function $f:(y,\tau
,\lambda)\mapsto f(y,\tau,\lambda)$ from $\mathbb{R}^{N}\times\mathbb{R}%
\times\mathbb{R}$ into $\mathbb{R}$ satisfies the properties:

\begin{enumerate}
\item[(i)] $f$ is measurable,

\item[(ii)] $f(y,\tau,0)=0$ for a.e. $y$ and $\tau$,

\item[(iii)] there exists a constant $c_{1}>0$ such that 
\begin{equation*}
\left\vert f(y,\tau,\lambda)-f(y,\tau,\mu)\right\vert \leq c_{1}\left\vert
\lambda-\mu\right\vert 
\end{equation*}
for almost $y,\tau$, and for all $\lambda,\mu\in\mathbb{R}$.
\end{enumerate}

From (ii) and (iii) above, we infer that

\begin{enumerate}
\item[(iv)] there exists a constant $c_{2}>0$ such that $\left\vert
f(y,\tau,\lambda)\right\vert \leq c_{2}\left( 1+\left\vert
\lambda\right\vert \right) $ for almost $y,\tau$, and for all $\lambda\in%
\mathbb{R}$.
\end{enumerate}

\textbf{(A3) Lipschitz continuity of }$g$. The function $g:(y,\tau,u
)\mapsto g(y,\tau,u )$ from $\mathbb{R}^{N}\times\mathbb{R}\times L^{2}(Q)$
into $L_{2}(\mathcal{U}, L^{2}(Q))$ satisfies:

\begin{enumerate}
\item[(i)] $g$ is measurable,

\item[(ii)] $g(y,\tau, 0)=0$ for a.e. in $y$ and $\tau$,

\item[(iii)] there exists a constant $c_{3}>0$ such that \newline
$\left\vert g(y,\tau,u)-g(y,\tau,v )\right\vert _{L_{2}(\mathcal{U},
L^{2}(Q))} \leq c_{3}\left\vert u -v \right\vert _{L^{2}(Q)} $ for a.e. in $%
y,\tau$, and for all $u ,v \in L^{2}(Q)$.
\end{enumerate}

Also as above, from (ii) and (iii) above, we infer that

\begin{enumerate}
\item[(iv)] there exists a constant $c_{4}>0$ such that $\left\vert
g(y,\tau, u )\right\vert _{L_{2}(\mathcal{U}, L^{2}(Q))} \leq c_{4}\left(
1+\left\vert u \right\vert _{L^{2}(Q)} \right) $ for a.e. in $y,\tau$ and
for all $u \in L^{2}(Q)$.
\end{enumerate}

In the following, we introduce the notion of probabilistic strong solution
for our problem (\ref{1}).

\begin{definition}
A probabilistic strong solution of the problem \emph{(\ref{1})} is a
stochastic process $u_{\varepsilon}$ such that:

\begin{enumerate}
\item[1)] $u_{\varepsilon}$, $u_{\varepsilon}^{\prime}$ are $\mathcal{F}_{t}$%
-measurable,

\item[2)] $u_{\varepsilon}\in L^{2}\left( \Omega,\mathcal{F},P;L^{\infty
}\left( 0,T;H_{0}^{1}(Q)\right) \right) $, and $u_{\varepsilon}^{\prime}\in
L^{2}\left( \Omega,\mathcal{F},P;L^{\infty}\left( 0,T;L^{2}(Q)\right)
\right) $,

\item[3)] $u_{\varepsilon}$ satisfies 
\begin{align*}
\left( u_{\varepsilon}^{\prime}(t),\phi\right) +\int_{0}^{t}\left(
A_{0}^{\varepsilon}\nabla u_{\varepsilon}(\tau),\nabla\phi\right) d\tau &
=(u^{1},\phi)+\int_{0}^{t}\left( f\left( \frac{x}{\varepsilon},\frac{\tau }{%
\varepsilon},u_{\varepsilon}(\tau)\right) ,\phi\right) d\tau \\
& +\sum_{k=1}^{\infty}\int_{0}^{t}\left( g\left( \frac{x}{\varepsilon},\frac{%
\tau}{\varepsilon},u_{\varepsilon}(\tau)\right) e_{k},\phi\right)
dW^{k}(\tau),
\end{align*}
for all $\phi\in\mathcal{C}_{0}^{\infty}(Q)$ and for a.e. $t\in(0,T)$.

\item[4)] $u_{\varepsilon}(0)=u^{0}$.
\end{enumerate}
\end{definition}

With this in mind, under conditions (A1)-(A3) and provided that $u^{0}\in
H_{0}^{1}(Q)$ and $u^{1}\in L^{2}(Q)$, the problem (\ref{1}) has a unique
strong solution $u_{\varepsilon}\in L^{2}(\Omega;\mathcal{C}%
([0,T];H_{0}^{1}(Q)))$ with $u_{\varepsilon}^{\prime}\in L^{2}(\Omega;%
\mathcal{C}([0,T];L^{2}(Q)))$. This existence and uniqueness result has been
achieved in \cite[Theorem 8.4, p. 189]{Chow}.

To simplify the notations, we set 
\begin{equation}
g^{\varepsilon}(.,.,u_{\varepsilon})(x,t)=g\left( \frac{x}{\varepsilon},%
\frac{t}{\varepsilon},u_{\varepsilon}(x,t)\right) \text{ and }%
g_{k}^{\varepsilon}(.,.,u_{\varepsilon})=g^{\varepsilon}(.,.,u_{\varepsilon
})e_{k}\text{ for }k\geq1.  \notag
\end{equation}
The question in this work is to determine the limit as $\varepsilon
\rightarrow0$, of the sequence of processes $u_{\varepsilon}$ under suitable
assumptions on the coefficients of (\ref{1}) (see assumption (\textbf{A4})
below). For a fixed probability space representing the random fluctuations
space, we shall assume that the coefficients of (\ref{1}) have various
deterministic behaviours ranging from the periodicity to the weak almost
periodicity. Our study therefore falls within the scope of the \textit{%
sigma-convergence for stochastic processes} that is a generalization of the
so-called sigma-convergence concept introduced in 2003 in \cite{Hom1}. One
of the chief merits of this work lies in the fact that we do not make use of
the concept of the spectrum of an algebra with mean value (viewed as a $%
C^{\ast}$-algebra as always considered before), thereby addressing one of
the main concerns of Applied Scientists whom will therefore be able to use
results arising from the use of sigma-convergence concept in homogenization
theory. Indeed the corrector problem (see (\ref{5.5'})) is in that case,
posed on the numerical space $\mathbb{R}^{N}$, instead of the abstract space
representing by the spectrum of the underlying algebra with mean value as it
was always the case in all the previous work dealing with the
sigma-convergence concept. In the same direction we refer to the preprint 
\cite{JW2015}, which is the first work in which we have initiated the
deterministic homogenization of PDEs without appealing to the spectrum of an
algebra with mean value.

Let us clearly state our main results here, in order to fix ideas.

Let $A$ be an algebra with mean value on $\mathbb{R}^{d}$ (integer $d\geq1$%
), that is, a closed subalgebra of the $\mathcal{C}^{\ast}$-algebra of
bounded uniformly continuous real-valued functions on $\mathbb{R}^{d}$,
which contains the constants, is translation invariant and is such that any
of its elements $u$ possesses a mean value $M(u)$ defined by 
\begin{equation*}
M(u)=\lim_{R\rightarrow\infty}\frac{1}{\left\vert B_{R}\right\vert }%
\int_{B_{R}}u(y)dy 
\end{equation*}
where $B_{R}=B(0,R)$ is the open ball in $\mathbb{R}^{d}$ centered at the
origin and of radius $R$. We denote by $B_{A}^{p}(\mathbb{R}^{d})$ ($1\leq
p<\infty$) the completion of $A$ with respect to the seminorm 
\begin{equation*}
\left\Vert u\right\Vert _{p}=(M(\left\vert u\right\vert ^{p}))^{\frac{1}{p}%
}=\left( \underset{R\rightarrow\infty}{\lim\sup}\frac{1}{\left\vert
B_{R}\right\vert }\int_{B_{R}}\left\vert u(y)\right\vert ^{p}dy\right) ^{%
\frac{1}{p}}. 
\end{equation*}
Before we may proceed forward, let us note that if $A=\mathcal{C}_{per}(Y)$,
the algebra of continuous $Y$-periodic functions on $\mathbb{R}^{d}$ ($%
Y=(0,1)^{d}$) then $B_{A}^{p}(\mathbb{R}^{d})=L_{per}^{p}(Y)$, the space of
functions in $L_{loc}^{p}(\mathbb{R}^{d})$ that are $Y$-periodic; in that
case, $M(u)=\int_{Y}u(y)dy$ for $u\in L_{per}^{p}(Y)$. Also, if $A=AP(%
\mathbb{R}^{d})$ (the continuous Bohr almost periodic functions; see e.g., 
\cite{Bohr}), then $B_{A}^{p}(\mathbb{R}^{d})$ is exactly the space of
Besicovitch almost periodic functions on $\mathbb{R}^{d}$; see \cite%
{Besicovitch}.

From an argument due to Besicovitch \cite{Besicovitch}, it is known that $%
B_{A}^{p}(\mathbb{R}^{d})\hookrightarrow L_{loc}^{p}(\mathbb{R}^{d})$, so
that $B_{A}^{p}(\mathbb{R}^{d})\hookrightarrow\mathcal{D}^{\prime}(\mathbb{R}%
^{d})$. With this embedding, we view any element of $B_{A}^{p}(\mathbb{R}%
^{d})$ as its representative in $L_{loc}^{p}(\mathbb{R}^{d})$, which allows
us to define the following space 
\begin{equation*}
B_{A}^{1,p}(\mathbb{R}^{d})=\{u\in B_{A}^{p}(\mathbb{R}^{d}):\nabla_{y}u%
\in(B_{A}^{p}(\mathbb{R}^{d}))^{d}\}, 
\end{equation*}
equipped with the seminorm 
\begin{equation*}
\left\Vert u\right\Vert _{1,p}=\left( \left\Vert u\right\Vert
_{p}^{p}+\left\Vert \nabla_{y}u\right\Vert _{p}^{p}\right) ^{\frac{1}{p}}\ \
(u\in B_{A}^{p}(\mathbb{R}^{d})), 
\end{equation*}
which is a complete seminormed space. Now, considering two algebras with
mean value $A_{y}$ and $A_{\tau}$ on $\mathbb{R}_{y}^{N}$ and $\mathbb{R}%
_{\tau}$ respectively, we define the product algebra with mean value $%
A=A_{y}\odot A_{\tau}$ as the closure in the $\sup$ norm in $\mathbb{R}^{N+1}
$, of the tensor product $A_{y}\otimes A_{\tau}$, and we assume that the
coefficients of (\ref{1}) satisfy:

\begin{itemize}
\item[(\textbf{A4})] For any $\lambda\in\mathbb{R}$ the functions $%
(y,\tau)\mapsto f(y,\tau,\lambda)$ and $(y,\tau)\mapsto g_{k}(y,\tau
,\lambda)$ belong to $B_{A}^{2}(\mathbb{R}_{y,\tau}^{N+1})$; for any $x\in%
\overline{Q} $ the matrix function $y\mapsto A_{0}(x,y)$ lies $%
(B_{A_{y}}^{2}(\mathbb{R}_{y}^{N}))^{N\times N}$. Here $g=(g_{k})_{1\leq
k\leq m}$.
\end{itemize}

\medskip

In what follows, we denote by the same letter $M$ the mean value on each of
the algebras $A_{y}$, $A_{\tau }$ and $A$ as well. This being so, let $%
(e_{j})_{1\leq j\leq N}$ denote the canonical basis in $\mathbb{R}^{N}$. For
each fixed $1\leq j\leq N$ and each $x\in Q$, consider the problem 
\begin{equation}
\left\{ 
\begin{array}{l}
\text{Find }\chi _{j}(x,\cdot )\in B_{\#A_{y}}^{1,2}(\mathbb{R}_{y}^{N})%
\text{ such that:} \\ 
\Div_{y}\left( A_{0}(x,\cdot )(e_{j}+\nabla _{y}\chi _{j}(x,\cdot ))\right)
=0\text{ in }\mathbb{R}^{N}.%
\end{array}%
\right.   \label{equ1}
\end{equation}%
Then we show that (\ref{equ1}) possesses at least a solution whose gradient
is unique in $B_{A_{y}}^{2}(\mathbb{R}_{y}^{N})^{N}$. With the functions $%
\chi _{j}$ at our disposal, we define the so-called \textit{homogenized}
coefficients as follows: 
\begin{align}
\widetilde{A}(x)& =M(A_{0}(x,\cdot )(I+\nabla _{y}\chi (x,\cdot )))\text{, }%
x\in Q  \label{equ3} \\
\widetilde{f}(r)& =M(f(\cdot ,\cdot ,r))\text{ and }\widetilde{g}%
(r)=M(g(\cdot ,\cdot ,r))\text{ for }r\in \mathbb{R}.  \notag
\end{align}%
Here $\chi (x,\cdot )=(\chi _{j}(x,\cdot ))_{1\leq j\leq N}$ and $I$ is the $%
N\times N$ identity matrix $(\delta _{ij})_{1\leq i,j\leq N}$, $\delta _{ij}$
the Kronecker delta. As we expect $u_{\varepsilon }$ to converge strongly,
only the direct average of $f$ and $g$ have to be taken, since we will not
deal with the product of two weakly convergent sequences at that level.

\begin{remark}
\label{r5.1}\emph{It can be easily checked straightforwardly that the
functions }$\widetilde{f}$\emph{\ and }$\widetilde{g}$\emph{\ are Lipschitz,
while the matrix }$\widetilde{A}(x)$\emph{\ is symmetric and satisfies
assumptions similar to those of }$A_{0}$\emph{\ (see \textbf{(A1)}).}
\end{remark}

With all this in mind, the first main result of the work is the following
theorem.

\begin{theorem}
\label{t1.1}Assume \emph{(A1)-(A4)} hold. For each $\varepsilon>0$ let $%
u_{\varepsilon}$ be the unique solution to \emph{(\ref{1})} on a given
stochastic system $(\Omega,\mathcal{F},\mathbb{P},W,\mathcal{F}^{t})$. Then
the sequence $u_{\varepsilon}$ converges in probability to $u_{0}$ in $%
L^{2}(Q_{T})$, where $u_{0}$ is the unique strong probabilistic solution to 
\begin{equation*}
\left\{ 
\begin{array}{l}
du_{0}^{\prime}-\Div\left( \widetilde{A}(x)\nabla u_{0}\right) dt=\widetilde{%
f}(u_{0})dt+\widetilde{g}(u_{0})dW\text{ in }Q_{T} \\ 
u_{0}=0\text{ on }\partial Q\times(0,T) \\ 
u_{0}(x,0)=u^{0}(x)\text{ and }u_{0}^{\prime}(x,0)=u^{1}(x)\text{ in }Q.%
\end{array}
\right. 
\end{equation*}
\end{theorem}

The problem (\ref{equ1}) above that has been used to define the homogenized
matrix $\widetilde{A}(x)$ is posed on the entire numerical set $\mathbb{R}%
^{N}$ and hence the numerical computation of $\widetilde{A}(x)$ is somewhat
difficult. We overcome this difficulty by considering rather approximate
coefficients defined as follows: for each $R>0$ set 
\begin{equation}
\widetilde{A}_{R}(x)=\frac{1}{\left\vert B_{R}\right\vert }%
\int_{B_{R}}A_{0}(x,y)(I+\nabla_{y}\chi_{R}(x,y))dy   \label{equ2}
\end{equation}
where $\chi_{R}(x,\cdot)=(\chi_{j,R}(x,\cdot))_{1\leq j\leq N}$ with $%
\chi_{j,R}(x,\cdot)\equiv u\in H_{0}^{1}(B_{R})$ being the unique solution
to the Dirichlet problem 
\begin{equation}
\Div_{y}\left( A_{0}(x,\cdot)(e_{j}+\nabla_{y}u)\right) =0\text{ in }B_{R}. 
\label{eq3}
\end{equation}
However in practice, the appropriate computational method used for this kind
of problems is the heterogeneous multiscale finite element method \cite%
{Assyr2005, Weinan2003} arising by choosing a sampling finite subset $%
\{x_{k}:1\leq k\leq d\}$ of $Q$ allowing to solve (\ref{eq3}) for a finite
family of the macroscopic variable (behaving in (\ref{eq3}) as a parameter) $%
x=x_{k}$. Therefore (\ref{eq3}) reads 
\begin{equation}
\Div_{y}\left( A_{0}(x_{k},\cdot)(e_{j}+\nabla_{y}u)\right) =0\text{ in }%
B_{R}.   \label{eq4}
\end{equation}
Based on (\ref{eq4}), we assume in the next result that the matrix function $%
x\mapsto A_{0}(x,\cdot)$ is constant (with respect to $x$), that is, $%
A_{0}(x,y)\equiv A_{0}(y)$ for any $(x,y)\in Q\times\mathbb{R}^{N}$.
Therefore 
\begin{equation}
\left\{ 
\begin{array}{l}
\widetilde{A}(x)\equiv\widetilde{A}=M(A_{0}(I+\nabla_{y}\chi)) \\ 
\widetilde{A}_{R}(x)\equiv\widetilde{A}_{R}=\left\vert B_{R}\right\vert
^{-1}\int_{B_{R}}A_{0}(y)(I+\nabla_{y}\chi_{R}(y))dy.%
\end{array}
\right.   \label{eq5}
\end{equation}

The second main result of the work reads as follows.

\begin{theorem}
\label{t1.2}Let $\widetilde{A}_{R}$ and $\widetilde{A}$ be defined by \emph{(%
\ref{eq5})}. Then $\widetilde{A}_{R}$ converges (as $R\rightarrow \infty$)
to the homogenized matrix $\widetilde{A}$.
\end{theorem}

The proof of Theorem \ref{t1.1} is given in Section \ref{sec4} while that of
Theorem \ref{t1.2} is done in Section \ref{sec6}.

Homogenization of stochastic partial differential equations with rapidly
oscillating coefficients was studied in \cite{Ichihara1, Ichihara2, MS2015,
MS2016-1, MS2016-2, WoukengArxiv, SAA13, Wang1, Wang2, Fu, Fu1, Jiang}, to
cite a few. The homogenization of hyperbolic stochastic partial differential
equations (SPDEs) is at its infancy as evidenced by the very few number of
published papers in that direction; see e.g. \cite{MS2015, MS2016-1,
MS2016-2}. In the three references above the authors deal with linear
hyperbolic SPDE associated to the operator 
\begin{equation*}
u^{\prime\prime}-\nabla\cdot\left( A_{0}\left( \frac{x}{\varepsilon}\right)
\nabla u\right) 
\end{equation*}
with periodic coefficients $A_{0}$ depending only on the fast variable $%
y=x/\varepsilon$. It is worth recalling that the study undertaken here is
the first one dealing with hyperbolic SPDEs beyond the periodic setting.

We emphasize that the use of the concept of sigma convergence allows us, not
only to extend the well-known results in the periodic setting to the almost
periodic framework and beyond, but also to take into account the microscopic
behaviour of the coefficients of the problem studied. This is very
important, as far as one deals with problems with strongly oscillating
coefficients, as it is the case here.

The rest of the work is organized as follows. Section \ref{sec2} deals with
some useful a priori estimates and the study of the tightness of the
sequence of probability laws of the solutions of (\ref{1}). In Section \ref%
{sec3}, we present the sigma-convergence for stochastic processes revisited.
Starting from the notion of algebras with mean value, we end with some
properties of the above concept. Finally in Section \ref{sec5} we present
some applications of Theorem \ref{t1.1}.

\section{A priori estimates and tightness property\label{sec2}}

In this section, we derive some a priori estimates and prove the tightness
of the probability measures generated by the solution of problem (\ref{1}).
Throughout $C$ will denote a generic constant independent of $\varepsilon$
that may vary from line to line.

\subsection{A priori estimates}

The following result gives some a priori estimates of the solution of
problem (\ref{1}).

\begin{lemma}
\label{m}Under the assumptions \textbf{(A1)-(A3)}, the solution $%
u_{\varepsilon }$ of problem \emph{(\ref{1})} satisfies the following
estimates 
\begin{equation}
\mathbb{E}\sup_{0\leq s\leq T}\Vert u_{\varepsilon }(s)\Vert
_{H_{0}^{1}(Q)}^{4}+\mathbb{E}\sup_{0\leq s\leq T}\Vert u_{\varepsilon
}^{\prime }(s)\Vert _{L^{2}(Q)}^{4}\leq C,  \label{ta1}
\end{equation}%
\begin{equation}
\mathbb{E}\left\vert u_{\varepsilon }^{\prime }(t)-\int_{0}^{t}g\left( \frac{%
x}{\varepsilon },\frac{\tau }{\varepsilon },u_{\varepsilon }(\tau )\right)
dW(\tau )\right\vert _{W^{1,2}([0,T];L^{2}(Q))}^{2}\leq C,  \label{ta2}
\end{equation}%
\begin{equation}
\mathbb{E}\left\vert \int_{0}^{t}g\left( \frac{x}{\varepsilon },\frac{\tau }{%
\varepsilon },u_{\varepsilon }(\tau )\right) dW(\tau )\right\vert
_{W^{\alpha ,4}([0,T];L^{2}(Q))}^{4}\leq C  \label{ta3}
\end{equation}%
for $\alpha \in \lbrack 0,\frac{1}{2})$ and for all $t\geq 0$.
\end{lemma}

\begin{remark}
See \cite{Flandoli} for the definitions and the properties of the spaces $%
W^{\alpha,p}([0,T];X)$ and $W^{1,2}([0,T],X)$ where $p>1$ and $\alpha
\in(0,1)$.
\end{remark}

\begin{proof}[Proof of Lemma \protect\ref{m}]
The proof of (\ref{ta1}) follows is standard and follows from the
application of the Ito's formula, Doob's inequality and the Gronwall's
lemma. The proof of (\ref{ta2}) follows from the relation 
\begin{equation}
u_{\varepsilon }^{\prime }(t)-\int_{0}^{t}g\left( \frac{x}{\varepsilon },%
\frac{\tau }{\varepsilon },u_{\varepsilon }(\tau )\right)
dW(s)=u^{1}+\int_{0}^{t}\Div(A_{0}^{\varepsilon }\nabla u_{\varepsilon
}(\tau ))d\tau +\int_{0}^{t}f\left( \frac{x}{\varepsilon },\frac{\tau }{%
\varepsilon },u_{\varepsilon }(\tau )\right) d\tau 
\end{equation}%
and the estimates (\ref{ta1}). (\ref{ta3}) follows from Lemma 2.1 of \cite%
{Flandoli}.
\end{proof}

\subsection{Tightness property of probability measures induced by the
solutions\label{subsec2.2}}

We consider the phase space 
\begin{equation*}
S=\mathcal{C}(0,T;\mathcal{U}_{0})\times L^{2}(0,T;L^{2}(Q))\cap
C(0,T;H^{-1}(Q))\times C(0,T;H^{-1}(Q)).
\end{equation*}%
We may think of the first component $S_{W}=\mathcal{C}(0,T;\mathcal{U}_{0})$
of this phase space as the set where the driving Brownian motion are defined
and the second component $S_{u}=L^{2}(0,T;L^{2}(Q))\cap \mathcal{C}%
(0,T;H^{-1}(Q))$ is the set where the solution $u_{\varepsilon }$ lives. The
third component $S_{u^{\prime }}=\mathcal{C}(0,T;H^{-1}(Q))$ is also the set
where the solution $u_{\varepsilon }^{\prime }$ lives.

We consider the probability measures 
\begin{equation}
\mu _{W}(.)=\mathbb{P}(W\in ..)\in Pr(\mathcal{C}(0,T;\mathcal{U}_{0}),
\end{equation}%
\begin{equation}
\mu _{u}^{\varepsilon }(.)=\mathbb{P}(u^{\varepsilon }\in .)\in
Pr(L^{2}(0,T;L^{2}(Q))\cap \mathcal{C}(0,T;H^{-1}(Q))),
\end{equation}%
\begin{equation}
\mu _{u^{\prime }}^{\varepsilon }(.)=\mathbb{P}(u_{\varepsilon }^{\prime
}\in .)\in Pr(\mathcal{C}(0,T;H^{-1}(Q))),
\end{equation}%
where $Pr(A)$ is the set of all probability measures on $(A,\mathcal{B}(A))$
for a complete separable metric space $A$. This defines a sequence of
probability measures 
\begin{equation}
\pi ^{\varepsilon }=\mu _{W}\times \mu _{u}^{\varepsilon }\times \mu
_{u^{\prime }}^{\varepsilon }
\end{equation}%
on the phase space $S$.

One of the main result of this section is the following theorem.

\begin{theorem}
\label{paul} The family of measures $\{\pi^{\varepsilon}\}$ is tight over
the phase space S.
\end{theorem}

\begin{proof}
Let $Z_{1}$ be the space of functions $\Phi (x,t)$ defined and measurable on 
$Q\times \lbrack 0,T]$ and such that 
\begin{equation*}
\sup_{0\leq t\leq T}\Vert \Phi (t)\Vert _{H_{0}^{1}(Q)}^{2}\leq C;\ \
\sup_{0\leq t\leq T}\Vert \Phi ^{\prime }(t)\Vert _{L^{2}(Q)}^{2}\leq C.
\end{equation*}%
$Z_{1}$ is a compact subset of $L^{2}(0,T;L^{2}(Q))$. Let $B_{R}^{1}$ be the
ball of radius $R>0$ in $Z_{1}$. Using the estimates (\ref{ta1}), we get 
\begin{equation}
\mu _{u}^{\varepsilon }((B_{R}^{1})^{c})\leq \frac{C}{R^{2}}.
\end{equation}%
According to Theorem 2.2 in \cite{Flandoli}, the following compact embedding
holds 
\begin{equation}
W^{1,2}([0,T];L^{2}(Q))\Subset \mathcal{C}(0,T;H^{-1}(Q)).
\end{equation}%
let $B_{R}^{2}$ the ball of radius $R$ in $W^{1,2}([0,T];L^{2}(Q))$. Based
on the estimates (\ref{ta1}), we also get 
\begin{equation}
\mu _{u}^{\varepsilon }((B_{R}^{2})^{c})\leq \frac{C}{R^{2}}.
\end{equation}%
We observe that $B_{R}^{1}\cap B_{R}^{2}$ is compact in $L^{2}(0,T;L^{2}(Q))%
\cap \mathcal{C}(0,T;H^{-1}(Q))$. and for any $R>0$ 
\begin{equation}
\mu _{u}^{\varepsilon }((B_{R}^{1}\cap B_{R}^{2})^{c})\leq \frac{C}{R^{2}}.
\end{equation}%
Using the estimates (\ref{ta2})-(\ref{ta3}) and the compact embeddings 
\begin{equation}
W^{1,2}(0,T;L^{2}(Q))\Subset \mathcal{C}(0,T;H^{-1}(Q)),
\end{equation}%
\begin{equation}
W^{\alpha ,4}(0,T;L^{2}(Q))\Subset \mathcal{C}(0,T;H^{-1}(Q)),
\end{equation}%
where $\alpha $ is such that $4\alpha >1$, we can prove as in \cite{Flandoli}
that $\mu _{u^{\prime }}^{\varepsilon }$ is tight in $\mathcal{C}%
(0,T;H^{-1}(Q))$. The tightness of $\mu _{W}^{\varepsilon }$ in $\mathcal{C}%
(0,T;\mathcal{U}_{0})$ is a classical result. This completes the proof of
the tightness of $\{\pi ^{\varepsilon }\}$ in $S$.
\end{proof}

Prokhorov's compactness result enables us to extract from $(\pi
^{\varepsilon })$ a subsequence $(\pi ^{\varepsilon _{n}})$ such that $\pi
^{\varepsilon _{n}}$ weakly converges to a probability measure $\pi $ on $S$%
. Skorokhod's theorem ensures the existence of a complete probability space $%
(\overline{\Omega },\overline{\mathcal{F}},\overline{\mathbb{P}})$ and
random variables $(W^{\varepsilon _{n}},u_{\varepsilon _{n}},u_{\varepsilon
_{n}}^{\prime })$ and $(\overline{W},u_{0},w)$ defined on $(\overline{\Omega 
},\overline{\mathcal{F}},\overline{\mathbb{P}})$ with values in $S$ such
that 
\begin{equation}
\text{the probability law of}~(W^{\varepsilon _{n}},u_{\varepsilon
_{n}},u_{\varepsilon _{n}}^{\prime })\text{ is }~\pi ^{\varepsilon _{n}}
\label{2.1}
\end{equation}%
\begin{equation}
\text{the probability law of }(\overline{W},u_{0},w)\text{ is }\pi ,
\label{2.2}
\end{equation}%
\begin{equation}
W^{\varepsilon _{n}}\rightarrow \overline{W}\text{ in }\mathcal{C}(0,T;%
\mathcal{U}_{0})~\overline{\text{ }\mathbb{P}}\text{-}a.s.,  \label{2.3}
\end{equation}%
\begin{equation}
u_{\varepsilon _{n}}\rightarrow \text{$u_{0}$ in }L^{2}(0,T;L^{2}(Q))\cap 
\mathcal{C}(0,T;H^{-1}(Q))\text{ }\overline{\mathbb{P}}\text{-}a.s.,
\label{2.4}
\end{equation}%
\begin{equation}
u_{\varepsilon _{n}}^{\prime }\rightarrow w\text{ in }\mathcal{C}%
(0,T;H^{-1}(Q))\text{ }\overline{\mathbb{P}}\text{-}a.s.  \label{2.5}
\end{equation}%
Arguing as in \cite{Sango}, we hereby note that $u_{\varepsilon
_{n}}^{\prime }$ is in fact a modification of the time derivative of $%
u_{\varepsilon _{n}}$. In fact let $v^{\varepsilon _{n}}$ be the time
derivative of $u_{\varepsilon _{n}}$ and write formally 
\begin{align}
u_{\varepsilon }(t)-u_{\varepsilon }(s)& =\int_{s}^{t}u_{\varepsilon
}^{\prime }(r)dr,  \notag \\
u_{\varepsilon _{n}}(t)-u_{\varepsilon _{n}}(s)& =\int_{s}^{t}v^{\varepsilon
_{n}}(r)dr,\ 0\leq s\leq t\leq T.  \notag
\end{align}%
In view of the fact that $(u_{\varepsilon },u_{\varepsilon _{n}}^{\prime })$
and $(u_{\varepsilon },u_{\varepsilon }^{\prime })$ have the same law, we
readily get, thanks to Fubini's theorem, the relation 
\begin{equation*}
\int_{s}^{t}\int_{A}u_{\varepsilon _{n}}^{\prime }(r)d\overline{\mathbb{P}}%
=\int_{s}^{t}\int_{A}v^{\varepsilon _{n}}(r)d\overline{\mathbb{P}}
\end{equation*}%
for any $A\in \overline{\mathcal{F}}$. Differentiating with respect to $t$,
we arrive at 
\begin{equation*}
\int_{A}u_{\varepsilon _{n}}^{\prime }(t)d\overline{\mathbb{P}}%
=\int_{A}v^{\varepsilon _{n}}(t)d\overline{\mathbb{P}}
\end{equation*}%
for any $t\in \lbrack 0,T]$. Thus $u_{\varepsilon _{n}}^{\prime
}=v^{\varepsilon _{n}}$ $\overline{\mathbb{P}}$-a.s..

We can see that $\{W^{\varepsilon _{n}}:\varepsilon _{n}\}$ is a sequence of
cylindrical Brownian motions evolving on $\mathcal{U}$. We let $\overline{%
\mathcal{F}}_{t}$ be the $\sigma $-algebra generated by $(\overline{W}%
(s),u_{0}(s),w(s))\text{, for }0\leq s\leq t$ and the null sets of $%
\overline{\mathcal{F}}$. We can show by arguing as in \cite{Bensoussan1}
that $\overline{W}$ is an $\overline{\mathcal{F}}_{t}$-adapted cylindrical
Wiener process evolving on $\mathcal{U}$. By the same argument as in \cite%
{Bensoussan2}, we can show that 
\begin{align}
& \left( u_{\varepsilon _{n}}^{\prime },\phi \right) +\int_{0}^{t}\left(
A_{0}^{\varepsilon _{n}}\nabla u_{\varepsilon _{n}}(\tau ),\nabla \phi
\right) d\tau =(u^{1},\phi )+\int_{0}^{t}\left( f\left( \frac{x}{\varepsilon
_{n}},\frac{\tau }{\varepsilon _{n}},u_{\varepsilon _{n}}(\tau )\right)
,\phi \right) d\tau   \notag \\
& +\left( \int_{0}^{t}g\left( \frac{x}{\varepsilon _{n}},\frac{\tau }{%
\varepsilon _{n}},u_{\varepsilon _{n}}(\tau )\right) dW_{\varepsilon
_{n}}(\tau ),\phi \right) ,  \label{ma}
\end{align}%
holds for $\phi \in H_{0}^{1}(Q)$ and for almost all $(\bar{\omega},t)\in 
\overline{\Omega }\times \lbrack 0,T]$.

Now, we derive a priori estimates for the sequences $u_{\varepsilon _{n}}$
and $u_{\varepsilon _{n}}^{\prime }$ obtained from the application of
Prokhorov and Skorokhod's compactness results. We know that they satisfy (%
\ref{ma}). Therefore they satisfy the a priori estimates corresponding to $%
u_{\varepsilon }$ and $u_{\varepsilon }^{\prime }$. Namely 
\begin{align}
& \bar{\mathbb{E}}\sup_{0\leq s\leq T}\Vert u_{\varepsilon _{n}}(s)\Vert
_{H_{0}^{1}(Q)}^{4}\leq C;~~\bar{\mathbb{E}}\sup_{0\leq s\leq T}\Vert
u_{\varepsilon _{n}}^{\prime }(s)\Vert _{L^{2}(Q)}^{4}\leq C,  \label{de3} \\
& \bar{\mathbb{E}}\left\vert u_{\varepsilon _{n}}^{\prime
}(t)-\int_{0}^{t}g\left( \frac{x}{\varepsilon _{n}},\frac{\tau }{\varepsilon
_{n}},u_{\varepsilon _{n}}(\tau )\right) dW^{\varepsilon _{n}}(\tau
)\right\vert _{W^{1,2}([0,T];L^{2}(Q))}^{2}\leq C,  \notag \\
& \bar{\mathbb{E}}\left\vert \int_{0}^{t}g\left( \frac{x}{\varepsilon _{n}},%
\frac{\tau }{\varepsilon _{n}},u_{\varepsilon _{n}}(\tau )\right)
dW^{\varepsilon _{n}}(\tau )\right\vert _{W^{\alpha
,4}([0,T];L^{2}(Q))}^{4}\leq C,
\end{align}%
for $\alpha \in \lbrack 0,\frac{1}{2})$ and for all $t\geq 0$.

Thus modulo extraction of a new subsequence (keeping the same notations) we
have 
\begin{align}
& u_{\varepsilon_{n}}\rightarrow\text{$u_{0}$}~\text{weak~ star}~\text{in}%
~L^{2}\left( \overline{\Omega},L^{\infty}(0,T;H_{0}^{1}(Q))\right) ,  \notag
\\
& u_{\varepsilon_{n}}^{\prime}\rightarrow w=u_{0}^{\prime}~\text{weak~ star}~%
\text{in}~L^{2}\left( \overline{\Omega},L^{\infty}(0,T;H^{-1}(Q))\right) , 
\notag
\end{align}
where $u_{0}^{\prime}$ is the time derivative of $u_{0}$. Next by (\ref{2.4}%
), (\ref{2.5}), (\ref{de3}) and Vitali's theorem, we have 
\begin{equation}
u_{\varepsilon_{n}}\rightarrow\text{$u_{0}$ in }L^{2}(\overline{\Omega}%
;L^{2}(0,T;L^{2}(Q)))   \label{2.6}
\end{equation}%
\begin{equation}
u_{\varepsilon_{n}}^{\prime}\rightarrow\text{$u_{0}^{\prime}$ in }L^{2}(%
\overline{\Omega};L^{2}(0,T;H^{-1}(Q))).   \label{2.6'}
\end{equation}
Hence for almost all $(\omega,t)\in\overline{\Omega}\times\lbrack0,T]$, we
get 
\begin{align}
u_{\varepsilon_{n}} & \rightarrow\text{$u_{0}$ in }L^{2}(Q),  \label{de4} \\
u_{\varepsilon_{n}}^{\prime} & \rightarrow\text{$u_{0}^{\prime}$ in }%
H^{-1}(Q)   \label{de5}
\end{align}
with respect to the measure $d\overline{\mathbb{P}}\otimes dt$.

\section{Sigma-convergence for stochastic processes\label{sec3}}

The concept of sigma-convergence relies on the notion of algebra with mean
value. Before we can state it, let us first and foremost\ set some
prerequisites about algebras with mean value.

\subsection{Algebras with mean value\label{subsec2.1}}

Let $\mathrm{BUC}(\mathbb{R}^{N})$ denote the Banach algebra of bounded
uniformly continuous real-valued functions defined on $\mathbb{R}^{N}$. For $%
u\in\mathrm{BUC}(\mathbb{R}^{N})$ we set 
\begin{equation*}
u_{R}=\frac{1}{\left\vert B_{R}\right\vert }\int_{B_{R}}u(y)dy 
\end{equation*}
where $B_{R}$ stands for the open ball in $\mathbb{R}^{N}$ of radius $R$
centered at the origin. We say that the function $u$ has a \textit{mean value%
} if the limit $\lim_{R\rightarrow\infty}u_{R}$ exists in $\mathbb{R}$. We
set 
\begin{equation}
M(u)=\lim_{R\rightarrow\infty}u_{R}\equiv\lim_{R\rightarrow\infty}\frac {1}{%
\left\vert B_{R}\right\vert }\int_{B_{R}}u(y)dy.   \label{3.0'}
\end{equation}
Let $u\in\mathrm{BUC}(\mathbb{R}^{N})$ and assume that $M(u)$ exists. Then 
\begin{equation}
u^{\varepsilon}\rightarrow M(u)\text{ in }L^{\infty}(\mathbb{R}^{N})\text{%
-weak}\ast\text{ as }\varepsilon\rightarrow0   \label{3.1'}
\end{equation}
where $u^{\varepsilon}\in\mathrm{BUC}(\mathbb{R}^{N})$ is defined by $%
u^{\varepsilon}(x)=u(x/\varepsilon)$ for $x\in\mathbb{R}^{N}$. This is an
easy consequence of the fact that the set of finite linear combinations of
the characteristic functions of open balls in $\mathbb{R}^{N}$ is dense in $%
L^{1}(\mathbb{R}^{N})$.

This being so, a closed subalgebra $A$ of $\mathrm{BUC}(\mathbb{R}^{N})$ is
said to be an algebra with mean value (algebra wmv, in short) on $\mathbb{R}%
^{N}$ if it contains the constants, is translation invariant ($%
\tau_{a}u=u(\cdot+a)\in A$ for any $u\in A$ and $a\in\mathbb{R}^{N}$) and
any of its elements possesses a mean value in the sense of (\ref{3.0'}).

To an algebra wmv $A$ are associated its regular subalgebras $A^{m}=\{\psi
\in\mathcal{C}^{m}(\mathbb{R}^{N}):$ $D_{y}^{\alpha}\psi\in A$ $\forall
\alpha=(\alpha_{1},...,\alpha_{N})\in\mathbb{N}^{N}$ with $\left\vert
\alpha\right\vert \leq m\}$ (where $D_{y}^{\alpha}\psi=\frac{\partial
^{\left\vert \alpha\right\vert }\psi}{\partial y_{1}^{\alpha_{1}}\cdot
\cdot\cdot\partial y_{N}^{\alpha_{N}}}$). Under the norm $\left\Vert
\left\vert u\right\vert \right\Vert _{m}=\sup_{\left\vert \alpha\right\vert
\leq m}\left\Vert D_{y}^{\alpha}\psi\right\Vert _{\infty}$, $A^{m}$ is a
Banach space. We also define the space $A^{\infty}=\{\psi\in\mathcal{C}%
^{\infty}(\mathbb{R}^{N}):$ $D_{y}^{\alpha}\psi\in A$ $\forall\alpha
=(\alpha_{1},...,\alpha_{N})\in\mathbb{N}^{N}\}$, a Fr\'{e}chet space when
endowed with the locally convex topology defined by the family of norms $%
\left\Vert \left\vert \cdot\right\vert \right\Vert _{m}$.

The concept of a product algebra wmv will be useful in our study. Let $A_{y}$
(resp. $A_{\tau}$) be an algebra wmv on $\mathbb{R}_{y}^{N}$ (resp. $\mathbb{%
R}_{\tau}$). We define the product algebra wmv $A_{y}\odot A_{\tau} $ as the
closure in $\mathrm{BUC}(\mathbb{R}^{N+1})$ of the tensor product $%
A_{y}\otimes A_{\tau}=\{\sum_{\text{finite}}u_{i}\otimes v_{i}:u_{i}\in A_{y}
$ and $v_{i}\in A_{\tau}\}$. This defines an algebra wmv on $\mathbb{R}^{N+1}
$.

We also define the notion of vector-valued algebra with mean value. Indeed,
let $F$ be a Banach space. We denote by \textrm{BUC}$(\mathbb{R}^{N};F)$ the
Banach space of bounded uniformly continuous functions $u:\mathbb{R}%
^{N}\rightarrow F$, endowed with the norm 
\begin{equation*}
\left\Vert u\right\Vert _{\infty}=\sup_{y\in\mathbb{R}^{N}}\left\Vert
u(y)\right\Vert _{F}
\end{equation*}
where $\left\Vert \cdot\right\Vert _{F}$ stands for the norm in $F$. Let $A$
be an algebra with mean value on $\mathbb{R}^{N}$. We denote by $A\otimes F$
the usual space of functions of the form 
\begin{equation*}
\sum_{\text{finite}}u_{i}\otimes e_{i}\text{ with }u_{i}\in A\text{ and }%
e_{i}\in F 
\end{equation*}
where $(u_{i}\otimes e_{i})(y)=u_{i}(y)e_{i}$ for $y\in\mathbb{R}^{N}$. With
this in mind, we define the vector-valued algebra wmv $A(\mathbb{R}^{N};F)$
as the closure of $A\otimes F$ in \textrm{BUC}$(\mathbb{R}^{N};F)$, and we
can check that $A_{y}\odot A_{\tau}=A_{y}(\mathbb{R}^{N};A_{\tau})=A_{\tau }(%
\mathbb{R};A_{y})$.

\bigskip Now, let $f\in A(\mathbb{R}^{N};F)$. Then, defining $\left\Vert
f\right\Vert _{F}$ by $\left\Vert f\right\Vert _{F}(y)=\left\Vert
f(y)\right\Vert _{F}$ ($y\in\mathbb{R}^{N}$), we have that $\left\Vert
f\right\Vert _{F}\in A$. Similarly we can define (for $0<p<\infty$) the
function $\left\Vert f\right\Vert _{F}^{p}$ and $\left\Vert f\right\Vert
_{F}^{p}\in A$. This allows us to define the Besicovitch seminorm on $A(%
\mathbb{R}^{N};F)$ as follows: for $1\leq p<\infty$, 
\begin{equation*}
\left\Vert f\right\Vert _{p,F}=\left( \lim_{R\rightarrow\infty}\frac {1}{%
\left\vert B_{R}\right\vert }\int_{B_{R}}\left\Vert f(y)\right\Vert
_{F}^{p}dy\right) ^{\frac{1}{p}}\equiv\left( M(\left\Vert f\right\Vert
_{F}^{p})\right) ^{\frac{1}{p}}\text{ for }f\in A(\mathbb{R}^{N};F) 
\end{equation*}
where $B_{R}$ is the open ball in $\mathbb{R}^{N}$ centered at the origin
and of radius $R$. Next, we define the Besicovitch space $B_{A}^{p}(\mathbb{R%
}^{N};F)$ as the completion of $A(\mathbb{R}^{N};F)$ with respect to $%
\left\Vert \cdot\right\Vert _{p,F}$. The space $B_{A}^{p}(\mathbb{R}^{N};F)$
is a complete seminormed subspace of $L_{loc}^{p}(\mathbb{R}^{N};F)$, and
the following hold true:

\begin{itemize}
\item[(\textbf{1)}] The space $\mathcal{B}_{A}^{p}(\mathbb{R}%
^{N};F)=B_{A}^{p}(\mathbb{R}^{N};F)/\mathcal{N}$ (where $\mathcal{N}=\{u\in
B_{A}^{p}(\mathbb{R}^{N};F):\left\Vert u\right\Vert _{p,F}=0\}$) is a Banach
space under the norm $\left\Vert u+\mathcal{N}\right\Vert _{p,F}=\left\Vert
u\right\Vert _{p,F}$ for $u\in B_{A}^{p}(\mathbb{R}^{N};F)$.

\item[(\textbf{2)}] The mean value $M:A(\mathbb{R}^{N};F)\rightarrow F$
extends by continuity to a continuous linear mapping (still denoted by $M$)
on $B_{A}^{p}(\mathbb{R}^{N};F)$ satisfying 
\begin{equation*}
L(M(u))=M(L(u))\text{ for all }L\in F^{\prime}\text{ and }u\in B_{A}^{p}(%
\mathbb{R}^{N};F). 
\end{equation*}
Moreover, for $u\in B_{A}^{p}(\mathbb{R}^{N};F)$ we have 
\begin{equation*}
\left\Vert u\right\Vert _{p,F}=\left[ M(\left\Vert u\right\Vert _{F}^{p})%
\right] ^{1/p}\equiv\left[ \lim_{R\rightarrow\infty}\frac{1}{\left\vert
B_{R}\right\vert }\int_{B_{R}}\left\Vert u(y)\right\Vert _{F}^{p}dy\right] ^{%
\frac{1}{p}}, 
\end{equation*}
and for $u\in\mathcal{N}$ one has $M(u)=0$.
\end{itemize}

It is to be noted that $\mathcal{B}_{A}^{2}(\mathbb{R}^{N};H)$ (when $F=H$
is a Hilbert space) is a Hilbert space with inner product 
\begin{equation}
\left( u,v\right) _{2}=M\left[ \left( u,v\right) _{H}\right] \text{ for }%
u,v\in\mathcal{B}_{A}^{2}(\mathbb{R}^{N};H),   \label{1.5}
\end{equation}
$(~,~)_{H}$ denoting the inner product in $H$.

Let us pay attention to the special case $F=\mathbb{R}$ for which $B_{A}^{p}(%
\mathbb{R}^{N}):=B_{A}^{p}(\mathbb{R}^{N};\mathbb{R})$ and $\mathcal{B}%
_{A}^{p}(\mathbb{R}^{N}):=\mathcal{B}_{A}^{p}(\mathbb{R}^{N};\mathbb{R})$.
The mean value extends in a natural way to $\mathcal{B}_{A}^{p}(\mathbb{R}%
^{N})$ as follows: for $u=v+\mathcal{N}\in\mathcal{B}_{A}^{p}(\mathbb{R}^{N})
$, we set $M(u):=M(v)$; this is well-defined since $M(v)=0$ for any $v\in 
\mathcal{N}$. The Besicovitch seminorm in $B_{A}^{p}(\mathbb{R}^{N})$ is
merely denoted by $\left\Vert \cdot\right\Vert _{p}$, and we have $B_{A}^{q}(%
\mathbb{R}^{N})\subset B_{A}^{p}(\mathbb{R}^{N})$ for $1\leq p\leq q<\infty$%
. From this last property one may naturally define the space $B_{A}^{\infty}(%
\mathbb{R}^{N})$ as follows: 
\begin{equation*}
B_{A}^{\infty}(\mathbb{R}^{N})=\{f\in\cap_{1\leq p<\infty}B_{A}^{p}(\mathbb{R%
}^{N}):\sup_{1\leq p<\infty}\left\Vert f\right\Vert _{p}<\infty\}\text{.}%
\;\;\;\;\;\;\;\;\; 
\end{equation*}
We endow $B_{A}^{\infty}(\mathbb{R}^{N})$ with the seminorm $\left[ f\right]
_{\infty}=\sup_{1\leq p<\infty}\left\Vert f\right\Vert _{p}$, which makes it
a complete seminormed space.

For $u=v+\mathcal{N}\in\mathcal{B}_{A}^{p}(\mathbb{R}^{N})$ ($1\leq
p\leq\infty$) and $y\in\mathbb{R}^{N}$, we define in a natural way the
translate $\tau_{y}u=v(\cdot+y)+\mathcal{N}$ of $u$, and as it can be seen
in \cite{NA2014, Deterhom}, this is well defined and induces a strongly
continuous $N$-parameter group of isometries $T(y):\mathcal{B}_{A}^{p}(%
\mathbb{R}^{N})\rightarrow\mathcal{B}_{A}^{p}(\mathbb{R}^{N})$ defined by $%
T(y)u=\tau_{y}u$. We denote by $\overline{\partial}/\partial y_{i}$ ($1\leq
i\leq N$) the infinitesimal generator of $T(y)$ along the $i$th coordinate
direction. We refer the reader to \cite{NA2014, Deterhom} for the properties
of $\overline{\partial}/\partial y_{i}$.

Now, let 
\begin{equation*}
\mathcal{B}_{A}^{1,p}(\mathbb{R}^{N})=\{u\in\mathcal{B}_{A}^{p}(\mathbb{R}%
^{N}):\frac{\overline{\partial}u}{\partial y_{i}}\in\mathcal{B}_{A}^{p}(%
\mathbb{R}^{N})\ \forall1\leq i\leq N\} 
\end{equation*}
and 
\begin{equation*}
\mathcal{D}_{A}(\mathbb{R}^{N})=\varrho(A^{\infty}) 
\end{equation*}
where $\varrho$ is the canonical mapping of $B_{A}^{p}(\mathbb{R}^{N})$ into 
$\mathcal{B}_{A}^{p}(\mathbb{R}^{N})=B_{A}^{p}(\mathbb{R}^{N})/\mathcal{N}$
defined by $\varrho(u)=u+\mathcal{N}$ ($\mathcal{N}$ being defined above).
It can be shown that $\mathcal{D}_{A}(\mathbb{R}^{N})$ is dense in $\mathcal{%
B}_{A}^{p}(\mathbb{R}^{N})$, $1\leq p<\infty$. We also have that $\mathcal{B}%
_{A}^{1,p}(\mathbb{R}^{N})$ is a Banach space under the norm 
\begin{equation*}
\left\Vert u\right\Vert _{\mathcal{B}_{A}^{1,p}}=\left( \left\Vert
u\right\Vert _{p}^{p}+\sum_{i=1}^{N}\left\Vert \frac{\overline{\partial}u}{%
\partial y_{i}}\right\Vert _{p}^{p}\right) ^{1/p}\ \ (u\in\mathcal{B}%
_{A}^{1,p}(\mathbb{R}^{N})) 
\end{equation*}
admitting $\mathcal{D}_{A}(\mathbb{R}^{N})$ as a dense subspace.

We end this first part with a further notion which will lead to the
definition of the space of correctors. A function $u\in\mathcal{B}_{A}^{1}(%
\mathbb{R}^{N})$ is said to be \emph{invariant} if for any $y\in\mathbb{R}%
^{N}$, $T(y)u=u$. It is immediate that the above notion of invariance is the
well-known one relative to dynamical systems. Therefore, an algebra with
mean value will be said to be \emph{ergodic} if every invariant function $u$
is constant in $\mathcal{B}_{A}^{1}(\mathbb{R}^{N})$. As in \cite{BMW} one
may show that $u\in\mathcal{B}_{A}^{1}$ is invariant if and only if $%
\overline {\partial}u/\partial y_{i}=0$ for all $1\leq i\leq N$. We denote
by $I_{A}^{p}(\mathbb{R}^{N})$ the set of $u\in\mathcal{B}_{A}^{p}(\mathbb{R}%
^{N})$ that are invariant. The set $I_{A}^{p}(\mathbb{R}^{N})$ is a closed
vector subspace of $\mathcal{B}_{A}^{p}(\mathbb{R}^{N})$ satisfying the
following important property:%
\begin{equation}
u\in I_{A}^{p}(\mathbb{R}^{N})\text{ if and only if }\frac{\overline{%
\partial }u}{\partial y_{i}}=0\text{ for all }1\leq i\leq N\text{.} 
\label{3.2'}
\end{equation}

We define $\mathcal{B}_{\#A}^{1,p}(\mathbb{R}^{N})$ to be the completion of $%
\mathcal{B}_{A}^{1,p}(\mathbb{R}^{N})/I_{A}^{p}(\mathbb{R}^{N})$ with
respect to the norm of the gradient 
\begin{equation*}
\left\Vert u\right\Vert _{\#,p}=\left\Vert \overline{\nabla}_{y}u\right\Vert
_{p}:=\left( \sum_{i=1}^{N}\left\Vert \frac{\overline{\partial}u}{\partial
y_{i}}\right\Vert _{p}^{p}\right) ^{1/p}\ \ \text{for }u\in\mathcal{B}%
_{A}^{1,p}(\mathbb{R}^{N})/I_{A}^{p}(\mathbb{R}^{N}), 
\end{equation*}
which makes it a Banach space. Moreover $\mathcal{B}_{\#A}^{1,p}(\mathbb{R}%
^{N})$ is reflexive ($1<p<\infty$) and further, the space $\{u\in\mathcal{D}%
_{A}(\mathbb{R}^{N})/I_{A}^{p}(\mathbb{R}^{N}):M(u)=0\}$ is dense in $%
\mathcal{B}_{\#A}^{1,p}(\mathbb{R}^{N})$.

\begin{remark}
\label{r3.1}\emph{Assume that the algebra wmv }$A$\emph{\ is ergodic. Then }$%
I_{A}^{p}(\mathbb{R}^{N})=\mathbb{R}$. Hence 

\begin{itemize}
\item[(1)] $\mathcal{B}_{A}^{1,p}(\mathbb{R}^{N})/I_{A}^{p}(\mathbb{R}%
^{N})=\{u\in\mathcal{B}_{A}^{1,p}(\mathbb{R}^{N}):M(u)=0\}$\emph{;}

\item[(2)] \emph{we may, instead of }$\mathcal{B}_{\#A}^{1,p}(\mathbb{R}^{N})
$\emph{, rather consider its dense subspace }$B_{\#A}^{1,p}(\mathbb{R}%
^{N})=\{u\in W_{loc}^{1,p}(\mathbb{R}^{N}):M(\nabla u)=0\}$\emph{\ equipped
with the gradient seminorm }$\left\Vert u\right\Vert _{B_{\#A}^{1,p}(\mathbb{%
R}^{N})}=\left\Vert \nabla _{y}u\right\Vert _{p}$\emph{. It can be seen from 
\cite[Theorem 3.12]{Casado} that any function in }$\mathcal{B}_{\#A}^{1,p}(%
\mathbb{R}^{N})$\emph{\ is an equivalence class of an element in }$%
B_{\#A}^{1,p}(\mathbb{R}^{N})$\emph{\ in the sense that two elements in }$%
B_{\#A}^{1,p}(\mathbb{R}^{N})$\emph{\ are identified by their gradients: }$%
u=v$\emph{\ in }$B_{\#A}^{1,p}(\mathbb{R}^{N})$\emph{\ if and only if }$%
\nabla u=\nabla v$\emph{\ in }$B_{A}^{p}(\mathbb{R}^{N})$\emph{, i.e., }$%
\left\Vert \nabla _{y}(u-v)\right\Vert _{p}=0$\emph{. As we shall see later
on, the latter space will be more convenient in practice.}
\end{itemize}
\end{remark}

\bigskip

For $u\in\mathcal{B}_{A}^{p}(\mathbb{R}^{N})$ (resp. $v=(v_{1},...,v_{N})\in(%
\mathcal{B}_{A}^{p}(\mathbb{R}^{N}))^{N}$), we define the gradient operator $%
\overline{\nabla}_{y}$ and the divergence operator $\overline {\nabla}%
_{y}\cdot$ by 
\begin{equation*}
\overline{\nabla}_{y}u:=\left( \frac{\overline{\partial}u}{\partial y_{1}}%
,...,\frac{\overline{\partial}u}{\partial y_{N}}\right) \text{ and\ }%
\overline{\nabla}_{y}\cdot v\equiv\overline{\Div}_{y}v:=\sum_{i=1}^{N}\frac{%
\overline{\partial}u}{\partial y_{i}}. 
\end{equation*}
Then\ the divergence operator sends continuously and linearly $(\mathcal{B}%
_{A}^{p^{\prime}}(\mathbb{R}^{N}))^{N}$ into $(\mathcal{B}_{A}^{1,p}(\mathbb{%
R}^{N}))^{\prime}$ and satisfies 
\begin{equation}
\left\langle \overline{\Div}_{y}u,v\right\rangle =-\left\langle u,\overline {%
\nabla}_{y}v\right\rangle \text{\ for }v\in\mathcal{B}_{A}^{1,p}(\mathbb{R}%
^{N})\text{ and }u=(u_{i})\in(\mathcal{B}_{A}^{p^{\prime}}(\mathbb{R}%
^{N}))^{N}\text{,}   \label{00}
\end{equation}
where $\left\langle u,\overline{\nabla}_{y}v\right\rangle
:=\sum_{i=1}^{N}M\left( u_{i}\frac{\overline{\partial}v}{\partial y_{i}}%
\right) $.

\subsection{Sigma-convergence for stochastic processes}

We follow here the presentation made in \cite{SAA13}. For the results stated
here, the reader is referred to \cite{SAA13} for the proofs and other
comments. However, for the sake of completeness, we recall some facts that
have been presented in the above mentioned work.

In all that follows, $Q$ is an open subset of $\mathbb{\mathbb{R}}^{N}$
(integer $N\geq1$), $T$ is a positive real number and $Q_{T}=Q\times(0,T)$.
Let $(\Omega,\mathcal{F},\mathbb{P})$ be a probability space. The
expectation on $(\Omega,\mathcal{F},\mathbb{P})$ will throughout be denoted
by $\mathbb{E}$. Let us first recall the definition of the Banach space of
bounded $\mathcal{F}$-measurable functions. Denoting by $F(\Omega)$ the
Banach space of all bounded functions $f:\Omega\rightarrow\mathbb{R}$ (with
the sup norm), we define $B(\Omega)$ as the closure in $F(\Omega)$ of the
vector space $H(\Omega)$ consisting of all finite linear combinations of the
characteristic functions $1_{X}$ of sets $X\in\mathcal{F}$. Since $\mathcal{F%
}$ is an $\sigma$-algebra, $B(\Omega)$ is the Banach space of all bounded $%
\mathcal{F}$-measurable functions. Likewise we define the space $B(\Omega;Z)$
of all bounded $(\mathcal{F},B_{Z})$-measurable functions $%
f:\Omega\rightarrow Z$, where $Z$ is a Banach space endowed with the $\sigma$%
-algebra of Borelians $B_{Z}$. The tensor product $B(\Omega)\otimes Z$ is a
dense subspace of $B(\Omega;Z)$: this follows from the obvious fact that $%
B(\Omega)$ can be viewed as a space of continuous functions over the \textit{%
gamma-compactification} \cite{Zhdanok} of the measurable space $(\Omega,%
\mathcal{F})$, which is a compact topological space. Next, for $X$ a Banach
space, we denote by $L^{p}(\Omega,\mathcal{F},\mathbb{P};X)$ the space of $X$%
-valued random variables $u$ such that $\left\Vert u\right\Vert _{X}$ is $%
L^{p}(\Omega,\mathcal{F},\mathbb{P})$-integrable.

Now, let $A_{y}$ and $A_{\tau}$ be two algebras wmv on $\mathbb{R}_{y}^{N}$
and $\mathbb{R}_{\tau}$ respectively, and let $A=A_{y}\odot A_{\tau}$ be
their product. We know that $A$ is the closure in $\mathrm{BUC}(\mathbb{R}%
_{y,\tau }^{N+1})$ of the tensor product $A_{y}\otimes A_{\tau}$. Points in $%
\Omega$ are as usual denoted by $\omega$. The generic element of $Q_{T}$ is
denoted by $(x,t)$ while any function in $A_{y}$ (resp. $A_{\tau}$ and $A $)
is of variable $y\in\mathbb{R}^{N}$ (resp. $\tau\in\mathbb{R}$ and $(y,\tau
)\in\mathbb{R}^{N+1}$). The mean value over $A_{y}$, $A_{\tau}$ and $A$ is
denoted by the same letter $M$. However, we shall often write $M_{y}$ (resp. 
$M_{\tau}$ and $M_{y,\tau}$) to differentiate them if there is any danger of
confusion. For a function $u\in L^{p}(\Omega,\mathcal{F},\mathbb{P}%
;L^{p}(Q_{T};\mathcal{B}_{A}^{p}(\mathbb{R}^{N+1})))$ we denote by $%
u(x,t,\cdot,\omega)$ (for any fixed $(x,t,\omega)\in Q_{T}\times\Omega$) the
function defined by 
\begin{equation*}
u(x,t,\cdot,\omega)(y,\tau)=u(x,t,y,\tau,\omega)\text{ for }(y,\tau )\in%
\mathbb{R}^{N+1}. 
\end{equation*}
Then $u(x,t,\cdot,\omega)\in\mathcal{B}_{A}^{p}(\mathbb{R}^{N+1})$, so that
the mean value of $u(x,t,\cdot,\omega)$ is defined accordingly.

Unless otherwise stated, random variables will always be considered on the
probability space $(\Omega,\mathcal{F},\mathbb{P})$. Finally, the letter $E$
will throughout denote exclusively an ordinary sequence $(\varepsilon
_{n})_{n\in\mathbb{N}}$ with $0<\varepsilon_{n}\leq1$ and $\varepsilon
_{n}\rightarrow0$ as $n\rightarrow\infty$. In what follows, the notations
are those of the preceding section.

\begin{definition}
\label{d3.1}\emph{A sequence of random variables }$(u_{\varepsilon
})_{\varepsilon >0}\subset L^{p}(\Omega ,\mathcal{F},\mathbb{P};L^{p}(Q_{T}))
$\emph{\ (}$1\leq p<\infty $\emph{) is said to }weakly $\Sigma $-converge%
\emph{\ in }$L^{p}(Q_{T}\times \Omega )$\emph{\ to some random variable }$%
u_{0}\in L^{p}(\Omega ,\mathcal{F},\mathbb{P};L^{p}(Q_{T};\mathcal{B}%
_{A}^{p}(\mathbb{R}^{N+1})))$\emph{\ if as }$\varepsilon \rightarrow 0$\emph{%
, we have } 
\begin{equation}
\begin{array}{l}
\iint_{Q_{T}\times \Omega }u_{\varepsilon }(x,t,\omega )v\left( x,t,\frac{x}{%
\varepsilon },\frac{t}{\varepsilon },\omega \right) dxdtd\mathbb{P} \\ 
\ \ \ \ \ \ \ \ \ \ \ \ \ \ \rightarrow \iint_{Q_{T}\times \Omega }M\left(
u_{0}(x,t,\cdot ,\omega )f(x,t,\cdot ,\omega )\right) dxdtd\mathbb{P}%
\end{array}
\label{3.1}
\end{equation}%
\emph{for every }$v\in L^{p^{\prime }}(\Omega ,\mathcal{F},\mathbb{P}%
;L^{p^{\prime }}(Q_{T};A))$\emph{\ (}$1/p^{\prime }=1-1/p$\emph{). We
express this by writing} $u_{\varepsilon }\rightarrow u_{0}$ in $%
L^{p}(Q_{T}\times \Omega )$-weak $\Sigma $.
\end{definition}

\begin{remark}
\label{r3.0}\emph{The above weak }$\Sigma $\emph{-convergence in }$%
L^{p}(Q_{T}\times \Omega )$\emph{\ implies the weak convergence in }$%
L^{p}(Q_{T}\times \Omega )$\emph{. One can show as in the usual setting of }$%
\Sigma $\emph{-convergence method \cite{Hom1} that each }$v\in L^{p}(\Omega ,%
\mathcal{F},\mathbb{P};L^{p}(Q_{T};A))$\emph{\ weakly }$\Sigma $\emph{%
-converges to }$\varrho \circ v$\emph{.}
\end{remark}

In order to simplify the notation, we will henceforth denote $L^{p}(\Omega,%
\mathcal{F},\mathbb{P};X)$ merely by $L^{p}(\Omega;X)$ if it is understood
from the context and there is no danger of confusion.

The following results can be found in \cite{SAA13} (see especially Theorems
2, 3 and 4 therein).

\begin{theorem}
\label{t3.1}Let $1<p<\infty$. Let $(u_{\varepsilon})_{\varepsilon\in
E}\subset L^{p}(\Omega;L^{p}(Q_{T}))$ be a sequence of random variables
verifying the following boundedness condition: 
\begin{equation*}
\sup_{\varepsilon\in E}\mathbb{E}\left\Vert u_{\varepsilon}\right\Vert
_{L^{p}(Q_{T})}^{p}<\infty. 
\end{equation*}
Then there exists a subsequence $E^{\prime}$ from $E$ such that the sequence 
$(u_{\varepsilon})_{\varepsilon\in E^{\prime}}$ is weakly $\Sigma$%
-convergent in $L^{p}(Q_{T}\times\Omega)$.
\end{theorem}

\begin{theorem}
\label{t3.2}Let $1<p<\infty$. Let $(u_{\varepsilon})_{\varepsilon\in
E}\subset L^{p}(\Omega;L^{p}(0,T;W_{0}^{1,p}(Q)))$ be a sequence of random
variables which satisfies the following estimate: 
\begin{equation*}
\sup_{\varepsilon\in E}\mathbb{E}\left\Vert u_{\varepsilon}\right\Vert
_{L^{p}(0,T;W^{1,p}(Q))}^{p}<\infty. 
\end{equation*}
Then there exist a subsequence $E^{\prime}$ of $E$ and a couple of random
variables $(u_{0},u_{1})$ with $u_{0}\in
L^{p}(\Omega;L^{p}(0,T;W_{0}^{1,p}(Q;I_{A}^{p}(\mathbb{R}_{y}^{N}))))$ and $%
u_{1}\in L^{p}(\Omega ;L^{p}(Q_{T};\mathcal{B}_{A_{\tau}}^{p}(\mathbb{R}%
_{\tau};\mathcal{B}_{\#A_{y}}^{1,p}(\mathbb{R}_{y}^{N}))))$ such that, as $%
E^{\prime}\ni\varepsilon\rightarrow0$, 
\begin{equation*}
u_{\varepsilon}\rightarrow u_{0}\ \text{in }L^{p}(Q_{T}\times\Omega )\text{%
-weak }\Sigma 
\end{equation*}
and%
\begin{equation}
\frac{\partial u_{\varepsilon}}{\partial x_{i}}\rightarrow\frac{\partial
u_{0}}{\partial x_{i}}+\frac{\overline{\partial}u_{1}}{\partial y_{i}}\text{%
\ in }L^{p}(Q_{T}\times\Omega)\text{-weak }\Sigma\text{, }1\leq i\leq N\text{%
.}   \label{3.2}
\end{equation}
\end{theorem}

The following modified version of Theorem \ref{t3.2} will be used below.

\begin{theorem}
\label{t3.3}Assume that the hypotheses of Theorem \emph{\ref{t3.2}} are
satisfied. Assume further that $A_{y}$ is ergodic. Finally suppose further
that $p\geq2$ and that there exist a subsequence $E^{\prime}$ from $E$ and a
random variable $u_{0}\in L^{p}(\Omega;L^{p}(0,T;W_{0}^{1,p}(Q)))$ such
that, as $E^{\prime}\ni\varepsilon\rightarrow0$, 
\begin{equation}
u_{\varepsilon}\rightarrow u_{0}\text{\ in }L^{2}(Q_{T}\times\Omega )\text{.}
\label{3.3}
\end{equation}
Then there exist a subsequence of $E^{\prime}$ (not relabeled) and a $%
\mathcal{B}_{A_{\tau}}^{p}(\mathbb{R}_{\tau};\mathcal{B}_{\#A_{y}}^{1,p}(%
\mathbb{R}_{y}^{N}))$-valued stochastic process $u_{1}\in
L^{p}(\Omega;L^{p}(Q_{T};\mathcal{B}_{A_{\tau}}^{p}(\mathbb{R}_{\tau};%
\mathcal{B}_{\#A_{y}}^{1,p}(\mathbb{R}_{y}^{N}))))$ such that \emph{(\ref%
{3.2})} holds when $E^{\prime}\ni\varepsilon\rightarrow0$.
\end{theorem}

We will also deal with the product of sequences. For that reason, we give
one further

\begin{definition}
\label{d3.2}\emph{A sequence }$(u_{\varepsilon})_{\varepsilon>0}$\emph{\ of }%
$L^{p}(Q_{T})$\emph{-valued random variables (}$1\leq p<\infty$\emph{) is
said to }strongly $\Sigma$-converge\emph{\ in }$L^{p}(Q_{T}\times\Omega )$%
\emph{\ to some }$L^{p}(Q_{T};\mathcal{B}_{A}^{p}(\mathbb{R}_{y,\tau}^{N+1}))
$\emph{-valued random variable }$u_{0}$\emph{\ if it is weakly }$\Sigma$%
\emph{-convergent towards }$u_{0}$\emph{\ and further satisfies the
following condition: }%
\begin{equation}
\left\Vert u_{\varepsilon}\right\Vert
_{L^{p}(Q_{T}\times\Omega)}\rightarrow\left\Vert u_{0}\right\Vert
_{L^{p}(Q_{T}\times\Omega ;\mathcal{B}_{AP}^{p}(\mathbb{R}%
_{y,\tau}^{N+1}))}.   \label{3.12}
\end{equation}
\emph{We denote this by }$u_{\varepsilon}\rightarrow u_{0}$\emph{\ in }$%
L^{p}(Q_{T}\times\Omega)$\emph{-strong }$\Sigma$\emph{.}
\end{definition}

\begin{remark}
\label{r3.2}\emph{Arguing as in \cite{Hom1} we can show that for any }$u\in
L^{p}(Q_{T}\times\Omega;A)$\emph{, the sequence }$(u^{\varepsilon
})_{\varepsilon>0}$\emph{\ is strongly }$\Sigma$\emph{-convergent to }$%
\varrho(u)$\emph{, where }$u^{\varepsilon}(x,t,\omega)=u(x/\varepsilon
,t/\varepsilon,\omega)$\emph{\ for }$(x,t,\omega)\in Q_{T}\times\Omega $%
\emph{.}
\end{remark}

The next result is of capital interest in the sequel (see the proof of
Proposition \ref{p4.1}). Its proof is copied on that of \cite[Theorem 6]%
{DPDE}.

\begin{theorem}
\label{t3.4}Let $1<p,q<\infty$ and $r\geq1$ be such that $1/r=1/p+1/q\leq1 $%
. Assume $(u_{\varepsilon})_{\varepsilon\in E}\subset
L^{q}(Q_{T}\times\Omega)$ is weakly $\Sigma$-convergent in $%
L^{q}(Q_{T}\times\Omega)$ to some $u_{0}\in L^{q}(Q_{T}\times\Omega;\mathcal{%
B}_{AP}^{q}(\mathbb{R}_{y,\tau}^{N+1}))$, and $(v_{\varepsilon})_{%
\varepsilon\in E}\subset L^{p}(Q_{T}\times\Omega)$ is strongly $\Sigma$%
-convergent in $L^{p}(Q_{T}\times\Omega)$ to some $v_{0}\in
L^{p}(Q_{T}\times\Omega;\mathcal{B}_{AP}^{p}(\mathbb{R}_{y,\tau}^{N+1}))$.
Then the sequence $(u_{\varepsilon}v_{\varepsilon})_{\varepsilon\in E}$ is
weakly $\Sigma$-convergent in $L^{r}(Q_{T}\times\Omega)$ to $u_{0}v_{0}$.
\end{theorem}

\section{Homogenization results and proof of Theorem \protect\ref{t1.1}\label%
{sec4}}

\subsection{Preliminaries}

The notations are those of the preceding sections. We remark that property (%
\ref{3.1}) in Definition \ref{d3.1} still holds true for $v\in B(\Omega ;%
\mathcal{C}(\overline{Q}_{T};B_{A}^{p^{\prime },\infty }(\mathbb{R}_{y,\tau
}^{N+1})))$ where $B_{A}^{p^{\prime },\infty }(\mathbb{R}_{y,\tau
}^{N+1})=B_{A}^{p^{\prime }}(\mathbb{R}_{y,\tau }^{N+1})\cap L^{\infty }(%
\mathbb{R}_{y,\tau }^{N+1})$ and as usual, $p^{\prime }=p/(p-1)$.

With this in mind, the use of the sigma-convergence method to solve the
homogenization problem for (\ref{1}) will be possible provided that the
assumption (A4) stated in Section \ref{sec1} holds true. We recall it here
for explicitness.

\begin{itemize}
\item[(\textbf{A4})] For any $\lambda \in \mathbb{R}$ the functions $f(\cdot
,\cdot ,\lambda )$ and $g_{k}(\cdot ,\cdot ,\lambda )$ belong to $B_{A}^{2}(%
\mathbb{R}_{y,\tau }^{N+1})$; the matrix $A_{0}(x,\cdot )\in (B_{A_{y}}^{2}(%
\mathbb{R}_{y}^{N}))^{N\times N}$ for any $x\in \overline{Q}$ and $k\geq 1$.
\end{itemize}

\begin{remark}
\label{r4.1}\emph{Hypothesis (A4) includes a variety of behaviours, ranging
from the periodicity to the weak almost periodicity.}
\end{remark}

The following important result is needed in order to pass to the limit in
the stochastic term.

\begin{lemma}
\label{l4.1}Let $(u_{\varepsilon })_{\varepsilon }$ be a sequence in $%
L^{2}(Q_{T}\times \Omega )$ such that $u_{\varepsilon }\rightarrow u_{0}$ in 
$L^{2}(Q_{T}\times \Omega )$ as $\varepsilon \rightarrow 0$ where $u_{0}\in
L^{2}(Q_{T}\times \Omega )$. Then for each positive integer $k$, we have, 
\begin{equation}
g_{k}^{\varepsilon }(\cdot ,\cdot ,u_{\varepsilon })\rightarrow g_{k}(\cdot
,\cdot ,u_{0})\text{ in }L^{2}(Q_{T}\times \Omega )\text{-weak }\Sigma \text{
as }\varepsilon \rightarrow 0\text{.}  \label{4.1}
\end{equation}%
We also have 
\begin{equation}
f^{\varepsilon }(\cdot ,\cdot ,u_{\varepsilon })\rightarrow f(\cdot ,\cdot
,u_{0})\text{ in }L^{2}(Q_{T}\times \Omega )\text{-weak }\Sigma \text{ as }%
\varepsilon \rightarrow 0\text{.}  \label{4.2}
\end{equation}
\end{lemma}

\begin{proof}
Let us first check (\ref{4.1}). For $u\in B(\Omega ;\mathcal{C}(\overline{Q}%
_{T}))^{N}$, the function $(x,t,y,\tau ,\omega )\mapsto g_{k}(y,\tau
,u(x,t,\omega ))$ lies in $B(\Omega ;\mathcal{C}(\overline{Q}%
_{T};B_{A}^{2,\infty }(\mathbb{R}_{y,\tau }^{N+1})))$, so that we have $%
g_{k}^{\varepsilon }(\cdot ,\cdot ,u)\rightarrow g_{k}(\cdot ,\cdot ,u)$ in $%
L^{2}(Q_{T}\times \Omega )$-weak $\Sigma $ as $\varepsilon \rightarrow 0$.
Next, since $B(\Omega ;\mathcal{C}(\overline{Q}_{T}))$ is dense in $%
L^{2}(Q_{T}\times \Omega )$, it can be easily shown that 
\begin{equation}
g_{k}^{\varepsilon }(\cdot ,\cdot ,u_{0})\rightarrow g_{k}(\cdot ,\cdot
,u_{0})\text{ in }L^{2}(Q_{T}\times \Omega )\text{-weak }\Sigma \text{ as }%
\varepsilon \rightarrow 0.  \label{Eq2}
\end{equation}%
Now, let $v\in L^{2}(\Omega ;L^{2}(Q_{T};A))$; then 
\begin{align*}
& \iint_{Q_{T}\times \Omega }g_{k}^{\varepsilon }(\cdot ,\cdot
,u_{\varepsilon })v^{\varepsilon }dxdtd\mathbb{P}-\iint_{Q_{T}\times \Omega
}M\left( g_{k}(\cdot ,\cdot ,u_{0})v\right) dxdtd\mathbb{P} \\
& =\iint_{Q_{T}\times \Omega }(g_{k}^{\varepsilon }(\cdot ,\cdot
,u_{\varepsilon })-g_{k}^{\varepsilon }(\cdot ,\cdot ,u_{0}))v^{\varepsilon
}dxdtd\mathbb{P}+\iint_{Q_{T}\times \Omega }g_{k}^{\varepsilon }(\cdot
,\cdot ,u_{0})v^{\varepsilon }dxdtd\mathbb{P} \\
& -\iint_{Q_{T}\times \Omega }M\left( g_{k}(\cdot ,\cdot ,u_{0})v\right)
dxdtd\mathbb{P}.
\end{align*}%
The convergence result (\ref{4.1}) therefore stems from the inequality 
\begin{equation*}
\left\vert \iint_{Q_{T}\times \Omega }(g_{k}^{\varepsilon }(\cdot
,u_{\varepsilon })-g_{k}^{\varepsilon }(\cdot ,u_{0}))v^{\varepsilon }dxdtd%
\mathbb{P}\right\vert \leq C\left\Vert u_{\varepsilon }-u_{0}\right\Vert
_{L^{2}(Q_{T}\times \Omega )^{N}}\left\Vert v^{\varepsilon }\right\Vert
_{L^{2}(Q_{T}\times \Omega )}
\end{equation*}%
associated to (\ref{Eq2}). The same lines of reasoning gives (\ref{4.2}).
\end{proof}

\subsection{Passage to the limit}

Let $(u_{\varepsilon_{n}})_{n}$ be the sequence determined in the Subsection %
\ref{subsec2.2} and satisfying Eq. (\ref{ma}). In view of (\ref{ma}) the
sequence $(u_{\varepsilon_{n}})_{n}$ also satisfies the a priori estimates (%
\ref{de3}). In view of (\ref{de3}) and by a diagonal process, one can find a
subsequence of $(u_{\varepsilon_{n}})_{n}$ (not relabeled) which weakly
converges in $L^{2}(\overline{\Omega};L^{2}(0,T;H_{0}^{1}(Q)))$ to the
function $u_{0}$, so that $u_{0}\in L^{2}(\overline{\Omega}%
;L^{2}(0,T;H_{0}^{1}(Q)))$. From Theorem \ref{t3.3}, we infer the existence
of a function $u_{1}\in L^{2}(Q_{T}\times\overline{\Omega};\mathcal{B}%
_{A_{\tau}}^{2}(\mathbb{R}_{\tau};\mathcal{B}_{\#A_{y}}^{1,2}(\mathbb{R}%
_{y}^{N})))$ such that the convergence result 
\begin{equation}
\frac{\partial u_{\varepsilon_{n}}}{\partial x_{i}}\rightarrow\frac{\partial
u_{0}}{\partial x_{i}}+\frac{\overline{\partial}u_{1}}{\partial y_{i}}\text{
in }L^{2}(Q_{T}\times\overline{\Omega})\text{-weak }\Sigma\text{ }(1\leq
i\leq N)   \label{4.3}
\end{equation}
holds when $\varepsilon_{n}\rightarrow0$. Next observe that $(u_{0},u_{1})\in%
\mathbb{F}_{0}^{1}$ where 
\begin{equation*}
\mathbb{F}_{0}^{1}=L^{2}(\overline{\Omega}\times(0,T);H_{0}^{1}(Q))\times
L^{2}(Q_{T}\times\overline{\Omega};\mathcal{B}_{A_{\tau}}^{2}(\mathbb{R}%
_{\tau};\mathcal{B}_{\#A_{y}}^{1,2}(\mathbb{R}_{y}^{N}))). 
\end{equation*}
The smooth counterpart of $\mathbb{F}_{0}^{1}$ is the space 
\begin{equation*}
\mathcal{F}_{0}^{\infty}=\left( B(\overline{\Omega})\otimes\mathcal{C}%
_{0}^{\infty}(Q_{T})\right) \times\left( B(\overline{\Omega})\otimes 
\mathcal{C}_{0}^{\infty}(Q_{T})\otimes\mathcal{E}\right) 
\end{equation*}
where $\mathcal{E}=\mathcal{D}_{A_{\tau}}(\mathbb{R}_{\tau})\otimes\left( 
\mathcal{D}_{A_{y}}(\mathbb{R}_{y}^{N})/\mathbb{R}\right) $ with $\mathcal{D}%
_{A}(\mathbb{R}^{N})=\varrho(A^{\infty})$, $\varrho$ being the canonical
surjection of $B_{A}^{2}(\mathbb{R}^{N})$ onto $\mathcal{B}_{A}^{2}(\mathbb{R%
}^{N})$ defined by $\varrho(u)=u+\mathcal{N}$ (see Subsection \ref{subsec2.1}%
). It can be easily shown that $\mathcal{F}_{0}^{\infty}$ is dense in $%
\mathbb{F}_{0}^{1}$ ($\mathbb{F}_{0}^{1}$ endowed with its natural topology
making it a Hilbert space).

Now, for $\mathbf{v}=(v_{0},v_{1})\in\mathbb{F}_{0}^{1}$ we set $\mathbb{D}%
v=\nabla v_{0}+\overline{\nabla}_{y}v_{1}$ where $\nabla$ stands for the
gradient operator with respect to $x\in Q$ and $\overline{\nabla}_{y}=(%
\overline{\partial}/\partial y_{i})_{1\leq i\leq N}$. We finally recall that
any $u=\varrho(v)\in\mathcal{B}_{A}^{2}(\mathbb{R}^{N})$ has a mean value
defined by $M(u):=M(v)$.

Bearing all this in mind, we have the following result.

\begin{proposition}
\label{p4.1}The couple $\mathbf{u}=(u_{0},u_{1})\in\mathbb{F}_{0}^{1}$
determined by \emph{(\ref{de4})} and \emph{(\ref{4.3})} above solves the
following variational problem: 
\begin{equation}
\left\{ 
\begin{array}{l}
-\int_{Q_{T}\times\overline{\Omega}}u_{0}^{\prime}\psi_{0}^{\prime }dxdtd%
\overline{\mathbb{P}}+\int_{Q_{T}\times\overline{\Omega}}M(A_{0}\mathbb{D}%
\mathbf{u}\cdot\mathbb{D}\Phi)dxdtd\overline{\mathbb{P}} \\ 
=\int_{Q_{T}\times\overline{\Omega}}M(f(\cdot,\cdot,u_{0})\psi_{0}dxdtd%
\overline{\mathbb{P}}+\int_{Q_{T}\times\overline{\Omega}}M(g(\cdot
,\cdot,u_{0})\psi_{0}dxd\overline{W}d\overline{\mathbb{P}} \\ 
\text{for all }(\psi_{0},\psi_{1})\in\mathcal{F}_{0}^{\infty}\text{.}%
\end{array}
\right.   \label{4.4}
\end{equation}
\end{proposition}

\begin{proof}
For the sake of simplicity, we drop the index $n$ from $\varepsilon_{n}$ and
henceforth write $\varepsilon$ instead of $\varepsilon_{n}$. This being so,
we set 
\begin{equation*}
\Phi_{\varepsilon}(x,t,\omega)=\psi_{0}(x,t,\omega)+\varepsilon\psi\left(
x,t,\frac{x}{\varepsilon},\frac{t}{\varepsilon},\omega\right) \text{\ \ (}%
(x,t,\omega)\in Q_{T}\times\overline{\Omega}\text{)}
\end{equation*}
where $(\psi_{0},\psi_{1}=\varrho_{y}(\psi))\in\mathcal{F}_{0}^{\infty} $
with $\psi$ being a representative of $\psi_{1}$ and $\varrho_{y}$ the
canonical surjection of $B_{A_{y}}^{2}(\mathbb{R}_{y}^{N})$ onto $\mathcal{B}%
_{A_{y}}^{2}(\mathbb{R}_{y}^{N})$. We recall that $\psi\in B(\overline{%
\Omega })\otimes\mathcal{C}_{0}^{\infty}(Q_{T})\otimes\lbrack
A_{\tau}^{\infty }\otimes(A_{y}^{\infty}/\mathbb{R})]$, $A_{y}^{\infty}/%
\mathbb{R}=\{\phi\in A_{y}^{\infty}:M_{y}(\phi)=0\}$. Then $%
\Phi_{\varepsilon}\in B(\overline {\Omega})\otimes\mathcal{C}%
_{0}^{\infty}(Q_{T})$, and taking $\Phi _{\varepsilon}$ as a test function
in the variational formulation of (\ref{1}) , we get 
\begin{align}
& -\int_{Q_{T}\times\overline{\Omega}}u_{\varepsilon}^{\prime}\frac {%
\partial\Phi_{\varepsilon}}{\partial t}dxdtd\overline{\mathbb{P}}+\int
_{Q_{T}\times\overline{\Omega}}A_{0}^{\varepsilon}\nabla
u_{\varepsilon}\cdot\nabla\Phi_{\varepsilon}dxdtd\overline{\mathbb{P}}
\label{4.5} \\
& =\int_{Q_{T}\times\overline{\Omega}}f^{\varepsilon}(\cdot,\cdot
,u_{\varepsilon})\Phi_{\varepsilon}dxdtd\overline{\mathbb{P}}\mathbb{+}%
\int_{Q_{T}\times\overline{\Omega}}g^{\varepsilon}(\cdot,\cdot,u_{%
\varepsilon })\Phi_{\varepsilon}dxdW^{\varepsilon}d\overline{\mathbb{P}}. 
\notag
\end{align}
Our aim is to pass to the limit in (\ref{4.5}). We shall consider each term
separately. But before we proceed forward, let us first observe that: 
\begin{align*}
\frac{\partial\Phi_{\varepsilon}}{\partial t} & =\frac{\partial\psi_{0}}{%
\partial t}+\left( \frac{\partial\psi}{\partial\tau}\right) ^{\varepsilon
}+\varepsilon\left( \frac{\partial\psi}{\partial t}\right) ^{\varepsilon} \\
\nabla\Phi_{\varepsilon} & =\nabla\psi_{0}+(\nabla_{y}\psi)^{\varepsilon
}+\varepsilon(\nabla\psi)^{\varepsilon},
\end{align*}
so that (up to a subsequence $\varepsilon\rightarrow0$) 
\begin{equation}
\frac{\partial\Phi_{\varepsilon}}{\partial t}\rightarrow\frac{\partial\psi
_{0}}{\partial t}\text{ in }L^{2}(\overline{\Omega}\times(0,T);H_{0}^{1}(Q))%
\text{-weak}   \label{4.6}
\end{equation}%
\begin{equation}
\nabla\Phi_{\varepsilon}\rightarrow\nabla\psi_{0}+\nabla_{y}\psi\text{ in }%
L^{2}(Q_{T}\times\overline{\Omega})^{N}\text{-strong }\Sigma\text{ (see
Remark \ref{r3.2}).}   \label{4.7}
\end{equation}
Also 
\begin{equation}
\Phi_{\varepsilon}\rightarrow\psi_{0}\text{ in }L^{2}(Q_{T}\times \overline{%
\Omega})\text{-strong.}   \label{4.8}
\end{equation}
Next, from (\ref{4.1}) we deduce that 
\begin{equation}
g^{\varepsilon}(\cdot,\cdot,u_{\varepsilon})\rightarrow M\left( g(\cdot
,\cdot,u_{0})\right) \text{ in }L^{2}(Q_{T}\times\overline{\Omega })\text{%
-weak }\text{ as }\varepsilon\rightarrow0\text{.}   \label{de1}
\end{equation}
Combining (\ref{de1}) with (\ref{4.8}), we get 
\begin{equation}
g^{\varepsilon}(\cdot,\cdot,u_{\varepsilon})\Phi_{\varepsilon}\rightarrow
M\left( g(.,.,u_{0})\right) \psi_{0}\text{ in }L^{2}(Q_{T}\times \overline{%
\Omega})\text{-weak }   \label{deu2}
\end{equation}
Similarly, we have 
\begin{equation}
f^{\varepsilon}(\cdot,\cdot,u_{\varepsilon})\Phi_{\varepsilon}\rightarrow
M(f(.,.,u_{0}))\psi_{0}\text{ in }L^{2}(Q_{T}\times\overline{\Omega })\text{%
-weak }.   \label{deu4}
\end{equation}
Now combining (\ref{deu2}) with (2.13) and arguing as in \cite{Bensoussan1},
we get 
\begin{equation}
\int_{Q_{T}\times\overline{\Omega}}g^{\varepsilon}(\cdot,\cdot,u_{%
\varepsilon })\Phi_{\varepsilon}dxdW^{\varepsilon}d\overline{\mathbb{P}}%
\rightarrow \int_{Q_{T}\times\overline{\Omega}}M\left(
g(\cdot,\cdot,u_{0})\right) \psi_{0}dxd\overline{W}d\overline{\mathbb{P}}. 
\label{deu5}
\end{equation}

Now, coming back to (\ref{4.5}) and considering there the second term of the
left-hand side, we note that we may use $A_{0}$ as test function for the
sigma-convergence (since it belongs to $\mathcal{C}(\overline{Q}%
;B_{A_{y}}^{\infty}(\mathbb{R}_{y}^{N}))$, which is contained in $B(%
\overline{\Omega };\mathcal{C}(\overline{Q}_{T};B_{A}^{\infty}(\mathbb{R}%
^{N+1})))$ and therefore get 
\begin{equation}
\int_{Q_{T}\times\overline{\Omega}}A_{0}^{\varepsilon}\nabla u_{\varepsilon
}\cdot\nabla\Phi_{\varepsilon}dxdtd\overline{\mathbb{P}}\rightarrow\int
_{Q_{T}\times\overline{\Omega}}M(A_{0}\mathbb{D}\mathbf{u}\cdot\mathbb{D}%
\Phi)dxdtd\overline{\mathbb{P}}.   \label{4.9}
\end{equation}
Indeed, we have (\ref{4.3}) and (\ref{4.7}), so that, by Theorem \ref{t3.4}, 
\begin{equation*}
\nabla u_{\varepsilon}\cdot\nabla\Phi_{\varepsilon}\rightarrow\mathbb{D}%
\mathbf{u}\cdot\mathbb{D}\Phi\text{ in }L^{1}\left( Q_{T}\times \overline{%
\Omega}\right) \text{-weak }\Sigma\text{.}
\end{equation*}
Hence, using the convergence results (\ref{4.6}), (\ref{2.6'}) (for the
first term of the left-hand side of (\ref{4.5})) and (\ref{4.9}), (\ref{deu4}%
), (\ref{deu5}) we are led at once to (\ref{4.4}).
\end{proof}

The problem (\ref{4.4}) is called the \textit{global homogenized} problem
for (\ref{1}).

\subsection{Homogenized problem}

The goal here is to derive the problem arising from the passage to limit (as 
$\varepsilon\rightarrow0$) whose $u_{0}\in L^{2}(\overline{\Omega}%
\times(0,T);H_{0}^{1}(Q))$ is the solution. For that, we first observe that (%
\ref{4.4}) is equivalent to the system made of (\ref{5.1}) and (\ref{5.2})
below: 
\begin{equation}
\int_{Q_{T}\times\overline{\Omega}}M(A_{0}\mathbb{D}\mathbf{u}\cdot 
\overline{\nabla}_{y}\psi_{1})dxdtd\overline{\mathbb{P}}=0\text{, all }%
\psi_{1}\in B(\overline{\Omega})\otimes C_{0}^{\infty}(Q_{T})\otimes 
\mathcal{E}   \label{5.1}
\end{equation}%
\begin{equation}
\left\{ 
\begin{array}{l}
-\int_{Q_{T}\times\overline{\Omega}}u_{0}^{\prime}\psi_{0}^{\prime }dxdtd%
\overline{\mathbb{P}}+\int_{Q_{T}\times\overline{\Omega}}M(A_{0}\mathbb{D}%
\mathbf{u})\cdot\nabla\psi_{0}dxdtd\overline{\mathbb{P}} \\ 
\ \ =\int_{Q_{T}\times\overline{\Omega}}M(f(\cdot,\cdot,u_{0})\psi _{0}dxdtd%
\overline{\mathbb{P}}+\int_{Q_{T}\times\overline{\Omega}}M(g(\cdot,%
\cdot,u_{0})\psi_{0}dxd\overline{W}d\overline{\mathbb{P}} \\ 
\text{for all }\psi_{0}\in B(\overline{\Omega})\otimes\mathcal{C}%
_{0}^{\infty }(Q_{T})\text{.}%
\end{array}
\right.   \label{5.2}
\end{equation}

Let us first consider the problem (\ref{5.1}) and there choose $\psi
_{1}=\varrho (\psi )$ with 
\begin{equation}
\psi (x,t,y,\tau ,\omega )=\varphi (x,t)\phi (y)\chi (\tau )\eta (\omega )
\label{5.3}
\end{equation}%
with $\varphi \in \mathcal{C}_{0}^{\infty }(Q_{T})$, $\phi \in A_{y}^{\infty
}/\mathbb{R}$, $\chi \in A_{\tau }^{\infty }$ and $\eta \in B(\overline{%
\Omega })$. Then (\ref{5.1}) becomes 
\begin{equation}
M_{y}(A_{0}(x,\cdot )\mathbb{D}\mathbf{u}\cdot \nabla _{y}\phi )=0\text{,
all }\phi \in A_{y}^{\infty }/\mathbb{R}.  \label{5.4}
\end{equation}%
So for $\xi \in \mathbb{R}^{N}$ be freely fixed, consider the cell problem: 
\begin{equation}
\left\{ 
\begin{array}{l}
\text{Find }\pi (\xi )\in \mathcal{B}_{\#A_{y}}^{1,2}(\mathbb{R}_{y}^{N})%
\text{ such that:} \\ 
-\overline{\Div}_{y}\left( A_{0}(x,\cdot )(\xi +\overline{\nabla }_{y}\pi
(\xi ))\right) =0\text{ in }\mathbb{R}_{y}^{N}.%
\end{array}%
\right.   \label{5.5}
\end{equation}%
Instead of (\ref{5.5}) and in view of (\ref{5.4}), we may rather consider
the more convenient problem 
\begin{equation}
\left\{ 
\begin{array}{l}
\text{Find }\pi _{1}(\xi )\in B_{\#A_{y}}^{1,2}(\mathbb{R}_{y}^{N})\text{
(see Remark \ref{r3.1}-(2)) such that:} \\ 
-\Div_{y}\left( A_{0}(x,\cdot )(\xi +\nabla _{y}\pi _{1}(\xi ))\right) =0%
\text{ in }\mathbb{R}_{y}^{N}.%
\end{array}%
\right.   \label{5.5'}
\end{equation}%
Then it can be easily shown (using property (\ref{00}) and the assumption (%
\textbf{A1}) on $A_{0}$) that (\ref{5.5'}) possesses at least a solution $%
\pi _{1}(\xi )$ whose the equivalence class $\pi (\xi )$ in $\mathcal{B}%
_{\#A_{y}}^{1,2}(\mathbb{R}_{y}^{N})$ is unique and solves (\ref{5.5}). Now,
taking $\xi =\nabla u_{0}(x,t,\omega )$ (for a.e. $(x,t,\omega )\in
Q_{T}\times \Omega $) in (\ref{5.5}) and testing the resulting equation with 
$\psi $ as in (\ref{5.3}), and next integrating over $Q_{T}\times \overline{%
\Omega }$, we get (by the uniqueness of the solution to (\ref{5.5})) that 
\begin{equation}
u_{1}(x,t,y,\cdot ,\omega )=\pi (\nabla u_{0}(x,t,\omega ))(y)\text{ for
a.e. }(x,t,\omega )\in Q_{T}\times \Omega .  \label{5.6}
\end{equation}%
From which the uniqueness of $u_{1}$ defined as above and belonging to $%
L^{2}(Q_{T}\times \overline{\Omega };\mathcal{B}_{A_{\tau }^{2}}(\mathbb{R}%
_{\tau };\mathcal{B}_{\#A_{y}}^{1,2}(\mathbb{R}_{y}^{N})))$. Next for fixed $%
\xi \in \mathbb{R}^{N}$ and $r\in \mathbb{R}$ define the homogenized
coefficients as follows: 
\begin{align*}
\widetilde{A}(x)\xi & =M(A_{0}(x,\cdot )(\xi +\overline{\nabla }_{y}\pi (\xi
)))\text{, }x\in Q \\
\widetilde{f}(r)& =M(f(\cdot ,\cdot ,r))\text{ and }\widetilde{g}%
(r)=M(g(\cdot ,\cdot ,r)).
\end{align*}%
It is important to note that in view of the equality $\overline{\nabla }%
_{y}\pi (\xi )=\nabla _{y}\pi _{1}(\xi )$, if we take in the above
definition of $\widetilde{A}(x)\xi $ the special $\xi =e_{j}$ ($1\leq j\leq N
$) then we get the exact definition of $\widetilde{A}(x)$ given in Section %
\ref{sec1} (see (\ref{equ3}) therein). With this in mind, the next result
holds.

\begin{proposition}
\label{p4.2}The function $u_{0}$ solves the boundary value problem 
\begin{equation}
\left\{ 
\begin{array}{l}
du_{0}^{\prime}-\Div\left( \widetilde{A}(x)\nabla u_{0}\right) dt=\widetilde{%
f}(u_{0})dt+\widetilde{g}(u_{0})d\overline{W}\text{ in }Q_{T} \\ 
u_{0}=0\text{ on }\partial Q\times(0,T) \\ 
u_{0}(x,0)=u^{0}(x)\text{ and }u_{0}^{\prime}(x,0)=u^{1}(x)\text{ in }Q.%
\end{array}
\right.   \label{5.7}
\end{equation}
\end{proposition}

\begin{proof}
If in (\ref{5.2}) we replace $u_{1}$ by its expression in (\ref{5.6}) and
take therein $\psi_{0}(x,t,\omega)=\varphi(x,t)\phi(\omega)$ with $\phi\in B(%
\overline{\Omega})$ and $\varphi\in\mathcal{C}_{0}^{\infty}(Q_{T})$, we get
readily the variational formulation of (\ref{5.7}). The initial conditions
are getting accordingly.
\end{proof}

\begin{proposition}
\label{p5.2}Let $u_{0}$ and $u_{0}^{\ast}$ be two solutions of \emph{(\ref%
{5.7})} on the same probabilistic system $(\overline{\Omega },\overline{%
\mathcal{F}},\overline{\mathbb{P}},\overline{W},\overline {\mathcal{F}}^{t})$
with the same initial conditions $u^{0}$ and $u^{1}$. Then $%
u_{0}=u_{0}^{\ast}$ $\overline{\mathbb{P}}$-almost surely.
\end{proposition}

\begin{proof}
The functions $\tilde{f}$ and $\tilde{g}$ are Lipschitz. The matrix $%
\widetilde{A}(x)$ is symmetric and satisfies assumptions similar to those of 
$A_{0}$ (see (\textbf{A1})). We can then apply \cite[Theorem 8.4 of p. 189]%
{Chow} to prove the existence and uniqueness of strong the solution of
problem (\ref{1}).
\end{proof}

The proof of Theorem \ref{t1.1} that will follow shortly, combines the
pathwise uniqueness of of the solution of equation (\ref{1}) and the Gy\"{o}%
ngy-Krylov characterization of convergence in probability introduced in \cite%
{KRYLOV}. We recall here the precise result.

\begin{lemma}
\label{4}Let $X$ be a Polish space equipped with the Borel $\sigma $%
-algebra. A sequence of $X$-valued random variables $\{Y_{n},n\in \mathbb{N}%
\}$ converges in probability if and only if every subsequence of joint laws $%
\{\mu _{n_{k},m_{k}},k\in \mathbb{N}\}$, there exists a further subsequence
which converges weakly to a probability measure $\mu $ such that 
\begin{equation*}
\mu \left( (x,y)\in X\times X:x=y\right) =1.
\end{equation*}
\end{lemma}

Let us set $X=L^{2}(0,T;L^{2}(Q))\cap C(0,T; H^{-1}(Q))\times
C(0,T;H^{-1}(Q))$; $X_{1}=\mathcal{C}(0,T;\mathcal{U}_{0})$; $X_{2}=X\times
X\times X_{1}$. For $S\in\mathcal{B}(X)$, we set $\pi^{\varepsilon}(S)=%
\mathbb{P}\left( (u_{\varepsilon},u_{\varepsilon}^{\prime})\in S\right) $.
For $S\in \mathcal{B}(X_{1})$, we set $\pi^{W}=\mathbb{P}\left( W\in
S\right) $.

Next, we define the joint probability laws: 
\begin{align}
\pi^{\varepsilon,\varepsilon^{\prime}} & =\pi^{\varepsilon}\times
\pi^{\varepsilon^{\prime}},  \notag \\
\nu^{\varepsilon,\varepsilon^{\prime}} & =\pi^{\varepsilon}\times
\pi^{\varepsilon^{\prime}}\times\pi^{W}.  \notag
\end{align}
The following tightness property is satisfied.

\begin{lemma}
\label{3}The collection $\{\nu^{\varepsilon,\varepsilon^{\prime}}:%
\varepsilon,\varepsilon^{\prime}>0\}$ is tight on $(X_{2},\mathcal{B}(X_{2}))
$.
\end{lemma}

\begin{proof}
The proof is similar to the one of Theorem \ref{paul}.
\end{proof}

\subsection{Proof of Theorem \protect\ref{t1.1}}

Lemma \ref{3} implies that there exists a subsequence from $\{\nu
^{\varepsilon_{j},\varepsilon_{j}^{\prime}}\}$ still denoted by $\{\nu
^{\varepsilon_{j},\varepsilon_{j}^{\prime}}$\} which converges to a
probability measure $\nu$ on $(X_{2},\mathcal{B}(X_{2}))$. By Skorokhod's
theorem, there exists a probability space $\left( \overline{\Omega},%
\overline{\mathcal{F}},\overline{\mathbb{P}}\right) $ on which a sequence $%
\left( (u_{\varepsilon_{j}},u_{\varepsilon_{j}}^{\prime}),(u_{\varepsilon
_{j}^{\prime}},u_{\varepsilon_{j}^{\prime}}^{\prime}),W^{j}\right) $ is
defined and converges almost surely in $X_{2}$ to a couple of random
variables $\left( (u_{0},u_{0}^{\prime}),(v_{0},v_{0}^{\prime}),\overline{W}%
\right) $. Furthermore, we have 
\begin{align}
\mathcal{L}\left( (u_{\varepsilon_{j}},u_{\varepsilon_{j}}^{\prime
}),(u_{\varepsilon_{j}^{\prime}},u_{\varepsilon_{j}^{\prime}}^{%
\prime}),W^{j}\right) & =\nu^{\varepsilon_{j},\varepsilon_{j}^{\prime}}, 
\notag \\
\mathcal{L}\left( (u_{0},u_{0}^{\prime}),(v_{0},v_{0}^{\prime}),\overline {W}%
\right) & =\nu.  \notag
\end{align}
Now let 
\begin{align}
Z_{j}^{u_{\varepsilon},u_{\varepsilon}^{\prime}} & =\left( u_{\varepsilon
_{j}},u_{\varepsilon_{j}}^{\prime},W^{j}\right) \text{ and }%
Z_{j}^{u_{\varepsilon^{\prime}},u_{\varepsilon^{\prime}}^{\prime}}=\left(
u_{\varepsilon_{j}^{\prime}},u_{\varepsilon_{j}^{\prime}}^{\prime},W^{j}%
\right) ,  \notag \\
Z^{(u_{0},u_{0}^{\prime})} & =(u_{0},u_{0}^{\prime},\overline{W})\text{ and }%
Z^{(v_{0},v_{0}^{\prime})}=(v_{0},v_{0}^{\prime},\overline{W}).  \notag
\end{align}
We can infer from the above argument that $\pi^{\varepsilon_{j},\varepsilon
_{j}^{\prime}}$ converges to a measure $\pi$ such that 
\begin{equation*}
\pi(.)=\overline{\mathbb{P}}\left( \left(
(u_{0},u_{0}^{\prime}),(v_{0},v_{0}^{\prime})\right) \in.\right) . 
\end{equation*}
As above, we can show that $Z_{j}^{u_{\varepsilon},u_{\varepsilon}^{\prime}}$
and $Z_{j}^{u_{\varepsilon^{\prime}},u_{\varepsilon^{\prime}}^{\prime}}$
satisfy (\ref{ma}) and that $Z^{(u_{0},u_{0}^{\prime})}$ and $%
Z^{(v_{0},v_{0}^{\prime})}$ satisfy (\ref{1}) on the same stochastic system $%
\left( \overline{\Omega},\overline{\mathcal{F}},\overline{\mathbb{P}},%
\overline {W},\overline{\mathcal{F}}^{t}_{1}\right) $ where $\overline{%
\mathcal{F}}^{t}_{1}$ is the filtration generated by the couple $\left(
(u_{0},u_{0}^{\prime}),(v_{0},v_{0}^{\prime}),\overline{W}\right) $. Since
we have the uniqueness result above, then we conclude that $u_{0}=v_{0}$ in $%
L^{2}(Q_{T})$; $u_{0}^{\prime}=v_{0}^{\prime}\text{ in }L^{2}(0,T;H^{-1}(Q))$%
. Therefore 
\begin{align}
& \pi\left( \left( \left( (x,y),(x^{\prime},y^{\prime})\right) \in X\times
X:(x,y)=(x^{\prime},y^{\prime})\right) \right)  \notag \\
& =\overline{\mathbb{P}}\left( (u_{0},u_{0}^{\prime})=(v_{0},v_{0}^{\prime })%
\text{ in }X\right)  \notag \\
& =1.  \notag
\end{align}
This fact together with Lemma \ref{4} imply that the original sequence $%
(u_{\varepsilon},u_{\varepsilon}^{\prime})$ defined on the original
probability $(\Omega,\mathcal{F},\mathbb{P}),\mathcal{F}^{t},W$ converges in
probability to an element $(u_{0},u_{0}^{\prime})$ in the topology of $X$.
This implies that the sequence $(u_{\varepsilon})$ converges in probability
to $u_{0}$ in $L^{2}(Q_{T})$ and $u_{\varepsilon}^{\prime}$ converges in
probability to $u_{0}^{\prime}$ in $L^{2}(0,T;H^{-1}(Q))$. By the passage to
the limit as in the previous subsection, it is not difficult to show that $%
u_{0}$ is the unique strong solution of (\ref{5.7}). This ends the proof of
the theorem.

\section{Approximation of homogenized coefficients and proof of Theorem 
\protect\ref{t1.2}\label{sec6}}

We assume that the notation is as in the preceding sections. In the
preceding section, we saw that the corrector problem is posed on the whole
of $\mathbb{R}^{N}$. However, if the coefficients of our problem are locally
periodic (say the function $y\mapsto A_{0}(x,y)$ is $Y$-periodic for each
fixed $x$, $Y=(-1/2,1/2)^{N}$), then this problem reduces to another one
posed on the bounded subset $Y$ of $\mathbb{R}^{N}$, and this yields
coefficients that are computable when $x$ is fixed. Contrasting with the
periodic setting, the corrector problem in the general deterministic
framework cannot be reduced to a problem on a bounded domain. Therefore,
truncations must be considered, particularly on large domains like $B_{R}$
(or $(-R,R)^{N}$ in practice) with appropriate boundary conditions. In that
case the homogenized coefficients are captured in the asymptotic regime. We
proceed exactly as in the random setting (see \cite{BP2004}).

We make a truncation on the ball $B_{R}$ ($R>0$) and impose linear Dirichlet
boundary condition on $\partial B_{R}$: 
\begin{equation}
-\nabla_{y}\cdot\left(
A_{0}(x,\cdot)(e_{j}+\nabla_{y}\chi_{j,R}(x,\cdot))\right) =0\text{ in }%
B_{R},\ \ \chi_{j,R}(x,\cdot)\in H_{0}^{1}(B_{R}).   \label{8.3}
\end{equation}
The following result is classical and the proof is omitted.

\begin{lemma}
\label{l8.1}Problem \emph{(\ref{8.3})} possesses a unique solution $u$ that
satisfies the estimate 
\begin{equation}
\left( \frac{1}{\left\vert B_{R}\right\vert }\int_{B_{R}}\left\vert
\nabla_{y}\chi_{j,R}(x,\cdot)\right\vert ^{2}dy\right) ^{\frac{1}{2}}\leq C%
\text{ for any }R\geq1   \label{i}
\end{equation}
where $C$ is a positive constant independent of $R$.
\end{lemma}

Let $\chi_{j,R}(x,\cdot)$ be the solution to (\ref{8.3}). As we saw in
Section \ref{sec1}, we may assume here that the matrix $A_{0}$ does not
depend on the macroscopic variable $x$, so that the functions $%
\chi_{j,R}(x,\cdot)$ are constant with respect to $x\in Q$, that is, $%
\chi_{j,R}(x,y)\equiv\chi _{j,R}(y)$. We define therefore the effective and
approximate effective matrices $\widetilde{A}$ and $\widetilde{A}_{R}$
respectively, as in (\ref{eq5}) (see Section \ref{sec1}).\ Here below, we
restate and prove Theorem \ref{t1.2}.

\begin{theorem}
\label{t8.1}The generalized sequence of matrices $\widetilde{A}_{R}$
converges, as $R\rightarrow\infty$, to the homogenized matrix $\widetilde{A}$%
.
\end{theorem}

\begin{proof}
We set $w_{j}^{R}(y)=\frac{1}{R}\chi_{j,R}(Ry)$ for $y\in B_{1}$ and
consider the rescaled version of (\ref{8.3}) whose $w_{j}^{R}$ is solution.
It reads as 
\begin{equation}
-\nabla_{y}\cdot(A_{0}(e_{j}+\nabla_{y}w_{j}^{R}))=0\text{ in }B_{1}\text{,
\ }w_{j}^{R}=0\text{ on }\partial B_{1}.   \label{8.6}
\end{equation}
Then (\ref{8.6}) possesses a unique solution $w_{j}^{R}\in H_{0}^{1}(B_{1})$
satisfying the estimate 
\begin{equation}
\left\Vert \nabla_{y}w_{j}^{R}\right\Vert _{L^{2}(B_{1})}\leq C\ \ \ (1\leq
j\leq N)   \label{8.7}
\end{equation}
where $C>0$ is independent of $R>0$. Based on (\ref{8.7}) and for a fixed $%
1\leq j\leq N$, let $w_{j}\in H_{0}^{1}(B_{1})$ be the weak limit in $%
H_{0}^{1}(B_{1})$ of a weakly convergent subnet $(w_{j}^{R^{\prime}})_{R^{%
\prime}}$ of $(w_{j}^{R})_{R}$. Then proceeding as in the proof of Theorem %
\ref{t1.1} (see especially the proof of (\ref{4.9}) therein), it is an easy
exercise to see that $w_{j}$ solves the equation 
\begin{equation}
-\nabla_{y}\cdot(\widetilde{A}(e_{j}+\nabla_{y}w_{j}))=0\text{ in }B_{1}%
\text{,}   \label{8.9}
\end{equation}
and further thanks to \cite[Theorem 5.2]{Jikov}, the convergence result (as $%
R^{\prime}\rightarrow\infty$) 
\begin{equation}
A_{0}(e_{j}+\nabla_{y}w_{j}^{R^{\prime}})\rightarrow\widetilde{A}%
(e_{j}+\nabla_{y}w_{j})\text{ in }L^{2}(B_{1})^{N}\text{-weak}   \label{8.10}
\end{equation}
is satisfied. From the ellipticity property of $A_{0}$ and the uniqueness of
the solution to (\ref{8.9}) in $H_{0}^{1}(B_{1})$, we deduce that $w_{j}=0$,
so that $w=(w_{1},...,w_{N})=0$. We infer that the whole sequence $%
(w_{j}^{R})_{R}$ weakly converges towards $0$ in $H_{0}^{1}(B_{1})$.
Therefore, integrating (\ref{8.10}) over $B_{1}$, we readily get (when
setting $w^{R}=(w_{1}^{R},...,w_{N}^{R})$)%
\begin{equation*}
\widetilde{A}_{R}=\frac{1}{\left\vert B_{1}\right\vert }\int_{B_{1}}A_{0}(I+%
\nabla_{y}w^{R})dy\rightarrow\frac{1}{\left\vert B_{1}\right\vert }%
\int_{B_{1}}\widetilde{A}(I+\nabla_{y}w)dy=\widetilde{A}
\end{equation*}
as $R\rightarrow\infty$, $I$ being denoting the identical $N\times N$%
-matrix. This completes the proof.
\end{proof}

\begin{remark}
\label{r8.1}\emph{We also define the approximate coefficients }%
\begin{equation*}
\widetilde{f}_{R}(r)=\frac{1}{\left\vert B_{R}\right\vert }%
\int_{B_{R}}f(y,\tau,r)dyd\tau\text{ and }\widetilde{g}_{R}(r)=\frac{1}{%
\left\vert B_{R}\right\vert }\int_{B_{R}}g(y,\tau,r)dyd\tau 
\end{equation*}
\emph{where here }$B_{R}$\emph{\ stands for the open ball in }$\mathbb{R}%
^{N}\times\mathbb{R}$\emph{\ centered at the origin and of radius }$R$\emph{%
. We have trivially }%
\begin{equation*}
\widetilde{f}_{R}(r)\rightarrow\widetilde{f}(r)\text{\emph{\ and }}%
\widetilde{g}_{R}(r)\rightarrow\widetilde{g}(r)\text{\emph{\ when }}%
R\rightarrow\infty. 
\end{equation*}
\bigskip
\end{remark}

\section{Some concrete applications of Theorem \protect\ref{t1.1}\label{sec5}%
}

In the preceding section we made Assumption (A4) under which the
homogenization of (\ref{1}) has been made possible. Here we give some
physical situations that lead to (A4). They are listed in the following
problems.

\subsection{Problem 1 (Stochastic periodic homogenization)}

Here we assume that the coefficients of the stochastic problem (\ref{1}) are
periodic, that is

\begin{itemize}
\item[(A4)$_{1}$] The functions $A_{0}(x,\cdot)$, $f(\cdot,\cdot,\lambda) $
and $g(\cdot,\cdot,\lambda)$ are periodic of period $1$ in each coordinate
for each $x\in\overline{Q}$ and $\lambda\in\mathbb{R}$.
\end{itemize}

Then, setting $Y=(0,1)^{N}$ and $\mathcal{T}=(0,1)$, this leads to (\textbf{%
A4}) with $A_{y}=\mathcal{C}_{per}(Y)$, $A_{\tau}=\mathcal{C}_{per}(\mathcal{%
T})$ and so $A=\mathcal{C}_{per}(Y\times\mathcal{T})$, where $\mathcal{C}%
_{per}(Y)$ is the Banach algebra of continuous $Y$-periodic functions
defined on $\mathbb{R}^{N}$ (similar definition for $\mathcal{C}_{per}(%
\mathcal{T})$ and $\mathcal{C}_{per}(Y\times\mathcal{T})$). In this case we
have $B_{A_{y}}^{p}(\mathbb{R}^{N})=L_{per}^{p}(Y)$, $B_{A}^{p}(\mathbb{R}%
^{N+1})=L_{per}^{p}(Y\times\mathcal{T})$ for $1\leq p\leq \infty$ and $%
B_{\#A_{y}}^{1,2}(\mathbb{R}^{N})=H_{\#}^{1}(Y)=\{u\in
H_{per}^{1}(Y):\int_{Y}udy=0\}$.

In this special case, the homogenization result reads as follows.

\begin{theorem}
\label{t6.1}Under Assumptions \emph{(A1)-(A3)} and \emph{(A4)}$_{1}$ the
sequence of solutions to \emph{(\ref{1})} converges in probability to the
solution of \emph{(\ref{5.7})} with 
\begin{align*}
\widetilde{A}(x) & =\int_{Y}A_{0}(x,y)(I+\nabla_{y}\chi(x,y))dy \\
\widetilde{f}(r) & =\iint_{Y\times\mathcal{T}}f(y,\tau,r)dyd\tau\text{ and }%
\widetilde{g}(r)=\iint_{Y\times\mathcal{T}}g(y,\tau,r)dyd\tau,
\end{align*}
$\chi(x,\cdot)=(\chi_{j}(x,\cdot))_{1\leq j\leq N}\in H_{\#}^{1}(Y)^{N}$
being defined as the solution of the problem 
\begin{equation*}
\nabla_{y}\cdot(A_{0}(x,\cdot)(e_{j}+\nabla_{y}\chi_{j}(x,\cdot)))=0\text{
in }Y. 
\end{equation*}
\end{theorem}

\begin{proof}
The above result stems from the characterization of the mean value in the
periodic case: if $u\in L_{per}^{p}(Y\times\mathcal{T})$ then $M(u)=\iint
_{Y\times\mathcal{T}}u(y,\tau)dyd\tau$.
\end{proof}

\subsection{Problem 2 (Stochastic almost periodic homogenization)}

The functions $A_{0}(x,\cdot)$, $f(\cdot,\cdot,\lambda)$ and $g(\cdot
,\cdot,\lambda)$ are assumed to be Besicovitch almost periodic \cite%
{Besicovitch}. We then get (A4) with $A_{y}=AP(\mathbb{R}^{N})$, $%
A_{\tau}=AP(\mathbb{R})$ and so $A=AP(\mathbb{R}^{N+1})$, where $AP(\mathbb{R%
}^{N})$ \cite{Besicovitch, Bohr} is the algebra of Bohr continuous almost
periodic functions on $\mathbb{R}^{N}$. In this case the mean value of a
function $u\in AP(\mathbb{R}^{N})$ can be obtained as the unique constant
belonging to the closed convex hull of the family of the translates $%
(u(\cdot+a))_{a\in\mathbb{R}^{N}}$; see e.g. \cite{Jacobs}.

\subsection{Problem 3}

Let $F$ be a Banach space. Let $\mathcal{B}_{\infty}(\mathbb{R}^{d};F)$
denote the space of all continuous functions $\psi\in\mathcal{C}(\mathbb{R}%
^{d};F)$ such that $\psi(\zeta)$ has a limit in $F$ as $\left\vert
\zeta\right\vert \rightarrow\infty$. When $F=\mathbb{R}$ we set $\mathcal{B}%
_{\infty }(\mathbb{R}^{d};\mathbb{R})\equiv\mathcal{B}_{\infty}(\mathbb{R}%
^{d})$. It is known that $\mathcal{B}_{\infty}(\mathbb{R}^{d})$ is an
algebra with mean value on $\mathbb{R}^{d}$ for which the mean value of any
function $u\in\mathcal{B}_{\infty}(\mathbb{R}^{d})$ is obtained as the limit
at infinity: 
\begin{equation*}
M(u)=\lim_{\left\vert y\right\vert \rightarrow\infty}u(y)\text{; see \cite%
{Hom1}.}
\end{equation*}

With this in mind, our goal here is to homogenize Problem (\ref{1}) under
the hypothesis

\begin{itemize}
\item[(A4)$_{2}$] $A_{0}(x,\cdot)\in(L_{per}^{2}(Y))^{N\times N}$ for any $%
x\in\overline{Q}$; $f(\cdot,\cdot,\lambda),g_{k}(\cdot,\cdot,\lambda )\in%
\mathcal{B}_{\infty}(\mathbb{R}_{\tau};L_{per}^{2}(Y))$ for all $\lambda\in%
\mathbb{R}$ and $k\ge1$.
\end{itemize}

It is an easy task to see that the appropriate algebras with mean value here
are $A_{y}=\mathcal{C}_{per}(Y)$ and $A_{\tau}=\mathcal{B}_{\infty}(\mathbb{R%
}_{\tau})$, so that (A4) holds true with $A=\mathcal{B}_{\infty }(\mathbb{R}%
_{\tau})\odot\mathcal{C}_{per}(Y)=\mathcal{B}_{\infty}(\mathbb{R}_{\tau};%
\mathcal{C}_{per}(Y))$.

\subsection{Problem 4}

With the notations of Problem 3, we replace here $L_{per}^{2}(Y)$ by $%
B_{AP}^{2}(\mathbb{R}_{y}^{N})$, the space of Besicovitch almost periodic
functions. Then (A4) is verified with $A_{y}=AP(\mathbb{R}^{N})$ and $%
A_{\tau }=\mathcal{B}_{\infty}(\mathbb{R}_{\tau})$, and hence $A=\mathcal{B}%
_{\infty }(\mathbb{R}_{\tau};AP(\mathbb{R}_{y}^{N}))$.

\subsection{Problem 5 (Stochastic asymptotic almost periodic homogenization)}

Let $\mathcal{B}_{\infty,AP}(\mathbb{R}^{N})=\mathcal{B}_{\infty}(\mathbb{R}%
^{N})+AP(\mathbb{R}^{N})$. We know that $\mathcal{B}_{\infty ,AP}(\mathbb{R}%
^{N})$ is an algebra with mean value on $\mathbb{R}^{N}$ with the property
that $\mathcal{B}_{\infty,AP}(\mathbb{R}^{N})=\mathcal{C}_{0}(\mathbb{R}%
^{N})\oplus AP(\mathbb{R}^{N})$ (direct and topological sum; see e.g. \cite%
{NA}) where $\mathcal{C}_{0}(\mathbb{R}^{N})$ stands for the space of those $%
u$ in $\mathrm{BUC}(\mathbb{R}^{N})$ that vanish at infinity. Since $%
\lim_{\left\vert y\right\vert \rightarrow\infty}u(y)=0$ for any $u\in%
\mathcal{C}_{0}(\mathbb{R}^{N})$, any element in $\mathcal{B}_{\infty ,AP}(%
\mathbb{R}^{N})$ is asymptotically an almost periodic function.

Bearing all this in mind, we aim at studying the homogenization problem for (%
\ref{1}) under the hypothesis:

\begin{itemize}
\item[(A4)$_{3}$] $A_{0}(x,\cdot)\in B_{\mathcal{B}_{\infty,AP}}^{2}(\mathbb{%
R}^{N})$ and $f(\cdot,\cdot,\lambda),g(\cdot,\cdot,\lambda)\in B_{AP}^{2}(%
\mathbb{R}_{\tau};B_{\mathcal{B}_{\infty,AP}}^{2}(\mathbb{R}_{y}^{N}))$, all 
$x\in\overline{Q}$ and $\lambda\in\mathbb{R}$.
\end{itemize}

We recall that here $B_{\mathcal{B}_{\infty,AP}}^{2}(\mathbb{R}^{N})$
denotes the Besicovitch space associated to the algebra wmv $\mathcal{B}%
_{\infty ,AP}(\mathbb{R}^{N})$. Assumption (A4)$_{3}$ leads to (A4) with $%
A_{y}=\mathcal{B}_{\infty,AP}(\mathbb{R}_{y}^{N})$, $A_{\tau}=AP(\mathbb{R}%
_{\tau })$ and hence $A=AP(\mathbb{R}_{\tau};\mathcal{B}_{\infty,AP}(\mathbb{%
R}_{y}^{N}))$.

\begin{remark}
\label{r6.1}\emph{1) Some other assumptions leading to (A4) are in order; 2)
The assumption of Problem 5 includes the special case of asymptotic periodic
homogenization in which }$A_{y}=\mathcal{B}_{\infty}(\mathbb{R}^{N})+%
\mathcal{C}_{per}(Y)$\emph{\ and }$A_{\tau}=\mathcal{C}_{per}(\mathcal{T})$%
\emph{, a self-contained problem; 3) It is worth noticing that }$\mathcal{B}%
_{\infty}(\mathbb{R}_{\tau};AP(\mathbb{R}_{y}^{N}))\neq \mathcal{B}_{\infty}(%
\mathbb{R}_{y}^{N})+AP(\mathbb{R}^{N})$\emph{.}
\end{remark}

\begin{acknowledgement}
\emph{The authors are very grateful to the referees for their comments and
suggestions, which have helped to considerably improve the quality of the
manuscript. The work has been partially done when J.L. Woukeng was visiting
the Abdus Salam International Centre for Theoretical Physics under the
Associate and Federation Schemes' grants. He gratefully acknowledges the
support of the Centre. He also acknowledges the support of the Humboldt
Foundation through the Georg Forster Fellowship for Experienced Researchers.}
\end{acknowledgement}

\end{document}